\documentclass{amsart}
\usepackage[utf8]{inputenc}
\usepackage{amssymb}
\usepackage{amsmath}
\usepackage{amsthm}
\usepackage{amscd}
\usepackage{mathrsfs}
\usepackage{bm}
\usepackage{color}
\usepackage{amsmath,amssymb,amsfonts}
\usepackage{tikz}
\usetikzlibrary{positioning}
\usetikzlibrary{svg.path}

\newcommand\crossing[4]{
\begin{tikzpicture}
\foreach \x in {#1,...,0}
\coordinate (\x) at (#1-\x,0) node [below] at (\x) {\x};
\foreach \x/\y in {#3}
{\draw (\x)--(#1-\x,\x+1);
\draw (\y)--(#1-\y, \x+1);
\draw (#1-\x,\x+1)--(#1-\y, \x+1);}
\foreach \x in {#4}
\draw (\x)--(#1-\x, #2+1);
\end{tikzpicture}
}

\theoremstyle{definition}
\newtheorem{dfn}{Definition}[section]
\theoremstyle{plain}
\newtheorem{thm}[dfn]{Theorem}

\newtheorem{prop}[dfn]{Proposition}
\newtheorem{cor}[dfn]{Corollary}
\newtheorem{lm}[dfn]{Lemma}
\theoremstyle{remark}
\newtheorem{exam}[dfn]{Example}
\newtheorem{rem}[dfn]{Remark}

\title{A dual and conjugate system for $q$-Gaussians for all $q$}

\author{Akihiro Miyagawa}
\address{Department of Mathematics, Kyoto University, Kitashirakawa Oiwake-cho, Sakyo-ku, 606-8502, Japan}
\email{miyagawa.akihiro.43v@st.kyoto-u.ac.jp}

\author{Roland Speicher}
\address{Saarland University, Department of Mathematics, D-66123 Saarbr\"ucken, Germany}
\email{speicher@math.uni-sb.de}

\begin{document}

\maketitle

\begin{abstract}
We provide a concrete formula for a dual system as well as for a conjugate system of $q$-Gaussians represented on the $q$-deformed Fock space. Moreover, using this formula, we prove the existence of a free Gibbs potential and that the non-microstates free Fisher information is finite for any $q$ with $-1<q<1$, which is an improvement on a previous result of Y. Dabrowski. We also indicate how our results can be extended to the more general setting of mixed $q_{ij}$-relations.
\end{abstract}

\section{Introduction}
The canonical commutation relations (CCR) are fundamental relations which are algebraically defined by 
$$ a_i a_j^* - a_j^* a_i = \delta_{i j} $$
While these relations describe Bosons in quantum mechanics, there are also the canonical anticommutation relations (CAR) which describe Fermions. These relations are defined by
$$ a_i a_j^* + a_j^* a_i = \delta_{i j}.$$
In \cite{BS91}, M. Bożejko and R. Speicher introduced $q$-CCRs which interpolate between CCR and CAR by using a parameter $-1\le q \le 1$, 
$$ a_i a_j^* - qa_j^* a_i = \delta_{i j}.$$
We will here only deal with the case of finitely many such operators, i.e., we will have $d\in\mathbb{N}$ and $1\leq i,j\leq d$.
These relations are represented as left creation and annihilation operators on the $q$-defomed Fock space, and studying the field operators ($q$-Gaussians) $A_i=a_i + a_i^*$ has attracted quite some interest. Especially, the $q$-Gaussian algebras $W^*(A)$, i.e., the von Neumann algebras generated by the $A_i$, have been studied for many years. One of the basic questions is whether and how those algebras depend on $q$.
The extreme cases $q=1$ and $q=-1$ are easy to understand and they are in any case different from the $q$ in the open interval $-1<q<1$. The central case $q=0$ is generated by free semicircular elements and free probability tools give then easily that this case is isomorphic to the free group factor. So the main question is whether the $q$-Gaussian algebras 
are, for $-1<q<1$, isomorphic to the free group factor. 

Over the years it has been shown that these algebras share many properties with the free group factors. For instance, for all $-1<q<1$ the $q$-Gaussian algebras are II$_1$-factors, non-injective, prime, and have strong solidity. Here is an incomplete list of papers proving these properties \cite{BKS97}, \cite{Sniady04}, \cite{Shlyakhtenko04}, \cite{Nou04}, \cite{Ricard05},  \cite{Avsec11}.   
There are also some random matrix models for $q$-Gaussians in \cite{Sniady01}, and \cite{PS21}.

A partial answer to the isomorphism problem was achieved by A. Guionnet and D. Shlyakhtenko \cite{GS14}, who proved that the $q$-Gaussian algebras are isomorphic to the free group factors for small $|q|$ (where the size of the interval depends on $d$ and goes to zero for $d\to\infty$). However, it is still open whether this is true for all $-1<q<1$.

In this paper, we compute a dual system and from this also a conjugate system for $q$-Gaussians. These notions were introduced by D. Voiculescu \cite{Voi98} in the context of free entropy and have turned out to carry important information about distributional properties of the considered operators and to have many implications for the related von Neumann algebras. 

In the following theorem we state our main result.

\begin{thm}\label{main}
Let $d\in\mathbb{N}$ be finite and $-1<q<1$ and consider corresponding $q$-Gaussians $A = (A_1,\ldots,A_d)$. 
Then there exists a normalized dual system and thus also a conjugate system for the $q$-Gaussians $A = (A_1,\ldots,A_d)$. Furthermore, the conjugate system is Lipschitz conjugate.
\end{thm}

Let us point out that the existence of a conjugate system for the $q$-Gaussians was shown for small $|q|$ by Y. Dabrowski \cite{Dabrowski10}, and Guionnet and Shlyakhtenko proved their isomorphism by using this result and the free monotone transport.
We remark that they consider right annihilation operators as a different version of dual systems, which are operators such that their commutators with $q$-Gaussians are equal to certain Hilbert-Schmidt operators. On the other hand, our approach starts from finding the concrete formula for dual systems which are operators whose commutators with $q$-Gaussians are exactly the orthogonal projection onto the vacuum vector. Our argument is based on the recursion induced by the definition of dual systems, and it allows us to give a precise combinatorial formula involving crossing partitions.

We want to call the attention of the reader to the fact that our formulas for the dual system and the conjugate operators contain a factor of the form $q^{m(m-1)/2}$ as coefficients for elememts in the $m$-particle space, in contrast to previous works where such coefficents were usually of the form $q^m$. Since all other exponents arising from norm estimates are only linear in $m$, this quadratic exponent in $m$ is in the end responsible for the fact that our estimates work for all $q$ in the interval $(-1,1)$.

Having the existence of conjugate systems for all $q$ with $-1<q<1$ has then, by general results, many consequences for all such $q$; like, for any $-1 < q <1$,
non-$\Gamma$ of $q$-Gaussian algebras, by \cite{Dabrowski10JFA}, or that any non-constant self-adjoint rational function over $q$-Gaussians has no atom in its distribution, by \cite{MSY18}, \cite{MSY19}. In Lemma 37 of \cite{Dabrowski10}, algebraic freeness of noncommutative power series over $q$-Gaussians is proved. 

There are also quite some applications of the fact that our conjugate system is Lipschitz conjugate. By \cite{Dabrowski10}, the existence of a Lipschitz conjugate system and Connes embeddability (which is given for our $q$-Gaussians, for all $q$) imply the maximality of the micro-states free entropy dimension. As a consequence of this or a direct application of Theorem 1.3 in \cite{DL16}, we can recover the fact that $W^*(A)$ has no Cartan subalgebra for any $-1<q<1$, which has been already shown by S. Avsec \cite{Avsec11} by other methods. Furthermore, the paper by M. Banna and T. Mai \cite{BM20} gives us Hölder continuity of cumulative distribution functions of noncommutative polynomials in the $q$-Gaussians. 

Let us collect in the following corollary the most important consequences of our result.

\begin{cor}
For all $-1<q<1$ we have the following properties.

1) The division closure of the $q$-Gaussians in the unbounded operators affiliated to $W^*(A_1,\dots,A_d)$ is isomorphic to the free field. This implies that any noncommutative rational function $r$ in $d$ non-commuting variables can be applied to the $q$-Gaussians, yielding a (possibly unbounded) operator $r(A_1,\dots,A_d)$. If $r$ is not the zero rational function, then this operator has trivial kernel; i.e., for any self-adjoint $r$ which is different from a constant the corresponding distribution has no atoms.

2) There is no non-zero noncommutative power series $\sum_{w \in[d]^*} \alpha_w A^w$ of radius of convergence $R>\|A_i\|$ such that $\sum_{w \in[d]^*} \alpha_w A^w = 0$.

3) For any self-adjoint noncommutative polynomial $Y=p(A_1,\dots,A_d)$ over $A_1,\ldots,A_d$, the cumulative distribution function $\mathcal{F}_Y$ of the distribution $Y$ is Hölder continuous with exponent $\frac{1}{2^{\mathrm{deg}Y}-1}$ where $\mathrm{deg}Y$ is the degree of $p$.  

4) The $q$-Gaussian operators have finite non-microstates free Fisher information and maximal microstates free entropy dimension,
$$\Phi^*(A_1,\dots,A_d)<\infty,\quad\text{and}\qquad \delta_0(A_1,\dots,A_d)=d.$$

5) $W^*(A_1,\dots,A_d)$ does not have property $\Gamma$, i.e., there is no non-trivial central sequence.

6) $W^*(A_1,\dots,A_d)$ does not have a Cartan subalgebra.

\end{cor}
 
Unfortunately, we are not able to use our result for adding anything to the isomorphism problem. However, the fact that the free entropy dimension is maximal for all $q$ in the whole interval is another indication that they might all be isomorphic to the free group factor.

    This paper consists, apart from this Introduction, of five parts. In Section \ref{Preliminaries},
    we give a short introduction to our main objects, namely $q$-Gaussians, conjugate systems, and dual systems; in particular, the general relation between conjugate and dual system is given.
    In Section \ref{OneVariable},
    we consider as a warm-up the one variable case $d=1$ and compute the dual operator and the conjugate variable. This case introduces also some basic qunatities in $q$-combinatorics. In Section \ref{MultiVariable}, we consider the multi-variable case and explain how to represent dual systems by special crossing partitions. In Section \ref{LipschitzConjugate}, we consider the noncommutative partial derivatives of our conjugate system and show the Lipschitz condition for those. In Section \ref{FreeGibbsPotential}, we discuss power series expansions of the conjugate system and the existence of free Gibbs potential. 
\section*{Acknowledgement}
This study started during the stay of the first author at Saarland University as a program of Kyoto University Top Global Project, which the second author hosted. During the visit, Y. Ueda asked the first author about this topic in an online workshop, which motivated us to study the dual system of $q$-Gaussians. We want to thank him for giving us a chance to explore this topic. The first author would like to express his gratitude to his Ph.D. supervisor B. Collins, for encouraging us to write this paper, and to S. Yin and J. Matsuda for checking the proof carefully. We are grateful to an anonymous referee for useful comments, which improved the clarity of the paper. 

A. Miyagawa was supported by JST, the establishment of university fellowships towards the creation of science technology innovation, Grant Number JPMJFS2123 and by JSPS Research Fellowships for Young Scientists, JSPS KAKENHI Grant Number JP 22J12186.

R. Speicher acknowledges support by the SFB-TRR 195 “Symbolic Tools in
Mathematics and their Application” of the German Research Foundation (DFG).

\section{Preliminaries}\label{Preliminaries}
\subsection{q-Gaussian algebras}
In this section, we introduce our main objects in a simple situation. Let $d \in \mathbb{N}$ be a positive integer and $H$ be a $d$-dimensional Hilbert space. We take an orthonormal basis $\{e_i\}_{i=1}^d$ of $H$. We consider the algebraic Fock space,
$$\mathcal{F}_{\mathrm{alg}}(H) = \bigoplus_{n=0}^{\infty} H^{\otimes n} $$
where we take the algebraic orthogonal sum of $H^{\otimes n}$ and where $H^{\otimes 0}$ is a one-dimensional Hilbert space $\mathbb{C} e_0$ with a unit vector (vacuum vector) $e_0$.
In order to describe a basis of $\mathcal{F}_{\mathrm{alg}}(H)$, we use the set $[d]^*$ of words in letters from the alphabet $[d] = \{1,\ldots,d\}$, with the empty word denoted by $0$. Namely we define $e_w$ for each $w= w_1 \cdots w_n \in [d]^*$ by $e_w = e_{w_1}\otimes e_{w_2} \otimes \ldots \otimes e_{w_n}.$ Note that the vacuum vector $e_0$ corresponds to the empty word $w=0$. Then all vectors in $\mathcal{F}_{\mathrm{alg}}(H)$ are represented by finite linear spans of $\{e_w\}_{w \in [d]^*}$
For each word $w \in [d]^*$, we define $|w|$ as the length of $w$. 

For a parameter $-1\le q\le 1$ and each $n \in \mathbb{Z}_{\ge 0}$, we define the operator $P^{(n)}: H^{\otimes n} \to H^{\otimes n} $ by $P^{(0)} = \mathrm{Id}$ and
$$P^{(n)} e_w = \sum_{\pi \in S_n} q^{|\pi|} e_{\pi(w)} $$
for each word of length $|w|=n$. $S_n$ is here the symmetric group of degree $n$ acting on $\{1,\ldots,n\}$ and $\pi(w) = w_{\pi(1)} \ldots w_{\pi(n)}$. We consider the following inner product on the algebraic Fock space:
$$ \langle \xi , \eta \rangle_q = \left \langle \xi , \bigoplus_{n=0}^{\infty} P^{(n)} \eta \right \rangle.$$

We remark that the operator $\bigoplus_{n=0}^{\infty} P^{\otimes n}$ is positive for any $-1\le q \le 1$ and strictly positive for any $-1<q<1$, by Proposition 1 in \cite{BS91}. Thus by taking the completion of the algebraic Fock space with respect to this inner product (where we first divide out the kernel of the inner product in the cases $q=-1,1$), we obtain a Hilbert space, the so-called $q$-deformed Fock space, denoted by $\mathcal{F}_q(H)$. We write its Hilbert norm by $\| \cdot \|_q$

Let us fix in the following $-1 < q <1$. Our main objects are left creation operators $\{l^*_i\}_{i=1}^d \in B(\mathcal{F}_q(H))$ defined by (see Lemma 2 and Lemma 4 in \cite{BS91}) 
$$ l^*_i e_w = e_{i w},$$
and the left annihilation operators $\{l_i\}_{i=1}^d$ defined by their adjoints $ l_i = (l_i^*)^*$(with respect to the $q$-inner product). 

Let us recall the concrete formula for $\{l_i\}_{i=1}^d$. From \cite{BS91}, we can see
$$ l_i e_{j_n j_{n-1} \cdots j_1} = \sum_{k=1}^n q^{n-k} \delta_{i j_k}e_{j_n\cdots \check{j_k} \cdots j_1} $$
where $\check{j_k}$ means to omit $j_k$ and where $\delta_{i j_k}$ is the Kronecker's delta. By using left creation and annihilation operators, $q$-Gaussians are defined as the corresponding field operators. 

\begin{dfn}\label{q-Gaussian}
We define the $q$-Gaussians $\{A_i\}_{i=1}^d$ by
$$A_i = l_i + l_i^* .$$
In addition, we define the $q$-Gaussian algebra $W^*(A)$ as the von Neumann subalgebra of $B(\mathcal{F}_q(H))$ generated by $A=(A_1,\ldots,A_d)$. 
\end{dfn}
By Theorem 4.4 in \cite{BS94}, the vacuum state 
$\tau_q(\cdot) = \langle \ \cdot \ e_0,e_0\rangle_q $
is tracial and faithful for $q$-Gaussian algebras. In addition, the Hilbert space obtained from the GNS-representation of the $q$-Gaussian algebra for the vacuum state is isomorphic to the $q$-deformed Fock space. 
We remark that for each $w \in [d]^*$, there exists a unique non-commutative polynomial $Q[w]$ over $q$-Gaussians such that $Q[w]e_0 = e_w$ (see, for example, Proposition 2.7 in \cite{BKS97}). 
This $Q[w]$ is given by 
$$Q[0]=1,\qquad  Q[j_1]=A_{j_1}$$
and
the following recursion for $j_{n+1},\ldots,j_1 \in [d]$:
$$
Q[j_{n+1}\cdots j_1] = A_{j_{n+1}}Q[j_n \cdots j_1]-\sum_{k=1}^n q^{n-k}Q[j_n \cdots \check{j_k} \cdots j_1]
$$

By linear extension we can thus identify elements in $\mathcal{F}_{\mathrm{alg}}(H)$ with operators in the algebra $C\langle A \rangle $ generated by our Gaussian operators.
Thus there are two different norms which we can apply to vectors $\eta\in \mathcal{F}_{\mathrm{alg}}(H)$: with $\Vert \eta\Vert_q$ we denote the Hilbert space norm of the vector $\eta$, whereas with $\Vert \eta\Vert$ we mean the operator norm of the corresponding operator, i.e., in particular
$\Vert e_w\Vert=\Vert Q[w]\Vert$.
It will be important to estimate those two norms against each other. Whereas 
$$\Vert\eta\Vert_q\leq \Vert \eta\Vert$$
is trivially true as the general comparison between $L_2$ and $L_\infty$ norms, the other direction needs more structure and is given by the following $q$-deformed version of Haagerup's inequality, proved by M. Bożejko (see Proposition 2.1 in \cite{Bozejko99}): for $m\geq 0$ and $\eta\in H^{\otimes m}$ one has
$$ \| \eta\| \le (m+1)C_{|q|}^{\frac{3}{2}} \|\eta\|_q,\qquad
\text{where} \qquad C_q^{-1} = \prod_{m=1}^{\infty} (1-q^m).$$

We will also use the free right annihilation operators $(r_1,\ldots,r_d)$ which are defined by $r_i e_{w j} = \delta_{i j} e_w$ ($w \in [d]^*$, $j \in [d]$) and $r_i e_0 = 0.$ Note that, only if we consider the full Fock space $\mathcal{F}_0(H)$, they are adjoint operators of the right creation operator which maps $e_w$ to $e_{w i}$ for $w \in [d]^*$. In the case $q \neq 0$, they are not adjoints of the right creation operators. The operators $l_i,l_i^*$ behave quite differently than the operators $r_i,r_i^*$; in particular, the latter are not the right version of the former. 
Whereas our operators $l_i$ and $l_j^*$ satisfy the $q$-commutation relations, this is not true for $r_i$ and $r_j^*$; also there is no nice concrete formula for the action of $r_j^*$ on the basic vectors $e_{j_n j_{n-1} \cdots j_1}$ (though, it is at least obvious that the $n$-particle space $H^{\otimes n}$ is mapped into the $(n+1)$-particle space $H^{\otimes (n+1)}$). Thus it is not directly clear how to determine the operator norm of those operators. However, by relying on results of Bożejko, we are able to give an estimate for this in the
following lemma.

\begin{lm}\label{FreeRightAnnihilation}
For $-1<q<1$, the free right annihilation operators $r_1,\ldots,r_d$ are bounded on $\mathcal{F}_q(H)$ with
$$\|r_i\| \le \frac{1}{\sqrt{w(q)}},\qquad
\text{where}\qquad w(q)^2 =(1-|q|^2)^{-1} \prod_{k=1}^{\infty} (1-|q|^k) (1+ |q|^k)^{-1}.$$ 
\end{lm}
\begin{proof}
Since $r_i$ respects the orthogonality between different tensor powers in the algebraic Fock space, it suffices to restrict for the norm estimate to a fixed tensor power $m+1 \in \mathbb{N}$; $r_i$
connects then elements in $H^{\otimes (m+1)}$ with elements in $H^{\otimes m}$. 
By Theorem 1 in \cite{Bozejko98}, we have
$$ P^{(m)} \otimes 1 \le w(q)^{-1} P^{(m+1)}$$
and we can estimate
\begin{eqnarray*}
\|r_i \sum_{|w| = m+1} \alpha_w e_w\|^2_q&=& \langle \sum_{|u| = m} \alpha_{ui}e_u, P^{(m)}\sum_{|v|=m}\alpha_{vi}e_v \rangle_{H^{\otimes m}}\\
&=& \sum_{|v| = m} \sum_{\pi \in S_m} \alpha_{\pi(v)i}\overline{\alpha_{vi}} q^{|\pi|} \\
&=& \langle \sum_{|u| = m+1}\alpha_u e_u , (P^{(m)} \otimes Q_i) \sum_{|v| = m+1}\alpha_v e_v \rangle_{H^{\otimes m+1}}
\end{eqnarray*}
where $Q_i \in B(H)$ is the orthogonal projection onto $\mathbb{C}e_i$.

Then we have $P^{(m)} \otimes Q_i \le P^{(m)} \otimes 1 \le w(q)^{-1}P^{(m+1)}$ and 
$$\|r_i \sum_{|w| = m+1} \alpha_w e_w\|^2_q \le w(q)^{-1}\| \sum_{|w| = m+1} \alpha_w e_w\|^2_q.   $$
\end{proof}

\subsection{Conjugate systems}\label{Conjugatesystem}
We recall the notion of a conjugate system, which was introduced in \cite{Voi98} in order to define the non-microstates free Fischer information $\Phi^*$.
Let $(\mathcal{M},\tau)$ be a tracial $W^*$-probability space which consists of a von Neumann algebra $\mathcal{M}$ and a faithful normal tracial state $\tau$. We consider the GNS-construction of $(\mathcal{M},\tau)$ and denote by $L^2(\mathcal{M},\tau)$ the Hilbert space of this representation.

In order to define the conjugate system, we consider the free difference quotients of non-commutative polynomials $$\partial_1,\ldots,\partial_d : \mathbb{C}\langle x_1,\ldots,x_d \rangle \to \mathbb{C}\langle x_1,\ldots,x_d \rangle \otimes \mathbb{C}\langle x_1,\ldots,x_d \rangle$$ 
which are linear operators determined by the Leibniz rule and  
$\partial_i x_j = \delta_{i j} 1\otimes 1.$  

\begin{dfn}
Let $X = (X_1,\ldots,X_d) \in \mathcal{M}^d$ be a $d$-tuple of self-adjoint operators. Then we say that $(\xi_1,\ldots,\xi_d) \in L^2(W^*(X), \tau)^d$ is a conjugate system for $X$ if they satisfy, for all $i=1,\dots,d$,
$$\tau(Q(X) \xi_i) = \tau \otimes \tau [\partial_i Q(X)]$$
for all $Q \in  \mathbb{C}\langle x_1,\ldots,x_d \rangle$ and where $Q(X)$ and $\partial_i Q(X)$ mean the evaluation of $Q$ and $\partial_i Q$, respectively, in $X$.  
\end{dfn}
We remark that the conjugate system is unique if it exists. When $X=(X_1,\ldots,X_d)$ are algebraically free, we can see $\partial_1,\ldots,\partial_d$ as unbounded operators from $L^2(W^*(X),\tau)$ to $ L^2(W^*(X),\tau) \otimes L^2(W^*(X),\tau)$ whose domains are $\mathbb{C}\langle X \rangle$. Then the existence of the conjugate system is equivalent to $1\otimes 1 \in \mathrm{dom}(\partial_i^*)$ for each $i \in [d]$, and the conjugate system is given by $\xi_i = \partial_i^*(1\otimes 1).$ 
If the conjugate system exists, we can see $\partial_1,\ldots,\partial_d$ are closable operators (see Corollary 4.2 in \cite{Voi98}). We say the conjugate system $(\xi_1,\ldots,\xi_d)$ is Lipschitz conjugate (see Definition 1 in \cite{Dabrowski10} or Section 2.4 in \cite{BM20}) if $\xi_i \in \mathrm{dom}(\overline{\partial_j})$ and $\overline{\partial_j} \xi_i \in W^*(X) \overline{\otimes} W^*(X)$ for any $i,j \in [d]$, where $\overline{\partial_j}$ is the closure of $\partial_j$ and we consider a von Neumann algebra tensor product for $W^*(X) \overline{\otimes} W^*(X)$. 

The non-microstates free Fisher information $\Phi^*(X_1,\ldots,X_d)$ of a tuple of self-adjoint elements $X=(X_1,\ldots,X_d)$ is defined by 
$$\Phi^*(X) = \begin{cases}\sum_{i=1}^d \| \xi_i \|^2& \text{if there exists a conjugate system  $(\xi_1,\ldots,\xi_d)$,}\\ \infty & \mathrm{otherwise}.  \end{cases}$$

The non-microstates free Fisher information is used to define the non-microstates free entropy $\chi^*(X)$ and the non-microstates free entropy dimension $\delta^*(X)$. We also remark that $\Phi^*(X)<\infty$ implies $-\infty < \chi^*(X)$ and $\delta^*(X)=d.$

A typical distribution with finite free Fisher information is given by $d$ free semicircular elements; this is the distribution of the $q$-Gaussians for the special case $q=0$.
In this case, the conjugate system for $A=(A_1,\ldots,A_d)$ is $A$ itself and $\Phi^*(A)=d.$ 
\subsection{Dual systems}\label{Dualsystem}
In \cite{Voi98}, D. Voiculescu also introduced the notion of a dual system. In this paper, we present the definition of dual systems according to \cite{Shlyakhtenko06}.
\begin{dfn}
Let $X = (X_1,\ldots,X_d) \in \mathcal{M}^d$ be a $d$-tuple of self-adjoint operators in a tracial $W^*$-probability space $(\mathcal{M},\tau)$.
Then a normalized dual system $(D_1,\ldots,D_d)$ is a tuple of unbounded operators on $L^2(W^*(X),\tau)$ whose domains contain $\mathbb{C}\langle X \rangle $ and which satisfy $D_i 1 = 0$, $1 \in \mathrm{dom}(D_i^*)$, and 
$$[D_i,X_j] = \delta_{i j}P $$
on $\mathbb{C}\langle X \rangle$ for any $i,j \in [d]$; where $P$ is the orthogonal projection onto $\mathbb{C} 1 \subset L^2(W^*(X),\tau).$  
\end{dfn}
In the free semicircular case (i.e., $q=0$), the free right annihilation operators $(r_1,\ldots,r_d)$ form a normalized dual system of $(A_1,\ldots,A_d)$. 

We put  ``normalized'' in the definition of dual systems since $D_i 1 = 0$ is an additional requirement compared to the original definition.
This assumption allows us to simplify our computations and is motivated by the following theorem, which tells us the relation between normalized dual systems and conjugate systems.

\begin{thm}[Theorem 1 in \cite{Shlyakhtenko06}]\label{shlyakhtenko}
For self-adjoint elements $X = (X_1,\ldots,X_d) \in \mathcal{M}^d$, the existence of a conjugate system $(\xi_1,\ldots,\xi_d)$ is equivalent to the existence of a normalized dual system $(D_1,\ldots,D_d)$. In this case we have for each $i \in [d]$
$$\xi_i =  D_i^* 1.$$
 \end{thm}

For the reader's convenience, let us give a proof of the relevant direction, from the dual system to the conjugate system. So let us compute $$\tau(X_{j_n}\cdots X_{j_1}D_i^*1)=\langle X_{j_n}\cdots X_{j_1}D_i^*1,1\rangle$$ for $j_n,\ldots,j_1 \in [d]$. Note that $X_1,\ldots,X_d$ are self-adjoint, and we have
\begin{align*}
\langle X_{j_n}\cdots X_{j_1}&D_i^* 1, \ 1\rangle
= \langle 1 ,D_i X_{j_1}\cdots X_{j_n}\rangle \\
  &= \langle 1,  X_{j_1}D_i X_{j_{2}} \cdots  X_{j_n}\rangle
   + \delta_{i j_1}\overline{\tau(X_{j_{2}} \cdots X_{j_n})} \\
   &= \cdots \\
   &=\sum_{k=1}^n \delta_{i j_k} \overline{\tau(X_{j_1}\cdots X_{j_{k-1}})\tau(X_{j_{k+1}}\cdots X_{j_n})} + \langle 1, X_{j_1}\cdots X_{j_n} D_i 1\rangle \\
   &= \sum_{k=1}^n \delta_{i j_k} \tau(X_{j_n}\cdots X_{j_{k+1}})\tau(X_{j_{k-1}}\cdots X_{j_1}) + \langle 1, X_{j_1}\cdots X_{j_n} D_i 1\rangle,
   \end{align*}
   where the last term is equal to $0$ since we required $D_i 1 = 0$. This implies that $(D_1^*1,\ldots,D_d^*1)$ forms the conjugate system.
   
   We remark that the condition $1 \in \mathrm{dom}(D_i^*)$ implies that $\mathbb{C}\langle X \rangle \subset \mathrm{dom}(D_i^*)$ and hence that $D_i$ is a closable operator. This can be seen by a similar computation as above, for $Q(X) \in \mathbb{C}\langle X \rangle $ and $ j_n,\ldots,j_1 \in [d]$:
   \begin{eqnarray*}
   \langle D_i Q(X) , X_{j_n}\cdots X_{j_1} \rangle &=& \langle D_i X_{j_1}\cdots X_{j_n}Q(X),1 \rangle \\
   &\ & -\sum_{k=1}^n \tau(X_{j_1} \cdots X_{j_{k-1}}) \langle P X_{j_{k+1}} \cdots X_{j_n} Q(X), 1 \rangle
   \end{eqnarray*}
where each term is a bounded operator with respect to $Q(X)$, since $X \in \mathcal{M}^d$ and $1 \in \mathrm{dom}(D_i^*).$ Thus $\mathbb{C}\langle X \rangle \subset \mathrm{dom}(D_i^*)$.  

In this paper, we will show the existence of a conjugate system from the existence of a dual system according to this theorem. In the $q$-Gaussian case, since $L^2(W^*(A),\tau_q)$ is isomorphic to the $q$-deformed Fock space $\mathcal{F}_q(H)$, we compute the dual system and the conjugate system as objects of $\mathcal{F}_q(H)$. The abstract embedding $W^*(A) \ni T \mapsto T 1 \in L^2(W^*(A),\tau_q)$ from the GNS-construction is here concretely realized  as the linear map $W^*(A) \ni T \mapsto T e_0 \in  \mathcal{F}_q(H)$.
 Via this map, we can identify $\mathbb{C}\langle A \rangle$ with $\mathcal{F}_{\mathrm{alg}}(H)$, since any $e_w$ can be written as $e_w = Q[w]e_0$ for some non-commutative polynomial $Q[w]$ in $A$.

\section{One variable case}\label{OneVariable}
 Consider first the case where $H = \mathbb{C}e$ with a unit vector $e$. Then $\{e_n\}_{n=0}^{\infty}$ defined by 
 $$e_n = e \otimes e \otimes \cdots \otimes e = e^{\otimes n}$$ 
 forms an orthonormal basis of $\mathcal{F}_{\mathrm{alg}}(H)$.
 Then the $q$-deformed inner product is determined by $\langle e_n,e_m\rangle_q = \delta_{n m}[n]_q!$ where
 $$[n]_q = \frac{1-q^{n}}{1-q}, \ \ [n]_q! = [n]_q\cdot [n-1]_q \cdots [2]_q \cdot [1]_q.  $$
 
Note that our $q$-Gaussian $A$ satisfies 
$$A e_n = e_{n+1} + [n]_q e_{n-1}.  $$

We would like to find an operator $D$ such that 
$D e_0 = 0$ and $[D,A] = P^{(0)}$.

Then $D$ needs to satisfy for $n \ge 1$
$$ D A e_n = A D e_n,$$
and then we have
$$ D e_{n+1} = -[n]_q D e_{n-1}+A D e_n. $$
Thus we obtain a recursion for $D e_n$ where $D e_0 =0$ and $D e_1 = e_0$.

\begin{exam}
For example, from this recursion we can compute
\begin{eqnarray*}
D e_2 &=& e_1\\
D e_3 &=& e_2 - q e_0 \\
D e_4 &=& e_3 - q(1+q) e_1 \\
D e_5 &=& e_4 - q (1+q+q^2) e_2 + q^3(1+q)e_0 \\
D e_6 &=& e_5 -q(1+q^2 + q^3)e_3 + q^3 (1+q) (1+q+q^2)e_1. 
\end{eqnarray*}
\end{exam}

\begin{rem}\label{q1}
In the case $q = -1$, we cannot define a linear operator $D$ by using the recursion, since $e_2=0$ in $\mathcal{F}_{-1}(H)$ while $D e_2 = e_1 \neq 0.$  
\end{rem}

Those examples suggest the following general 
explicit formula for $D e_n$.
\begin{prop}\label{ConjugateVariable}
We define
an unbounded operator $D$ with the domain $\mathcal{F}_{\mathrm{alg}}(H)$ by linear extension of
$De_0=0$ and
$$De_n =\sum_{k=1}^{\lceil \frac{n}{2} \rceil} (-1)^{k-1} q^{\frac{k(k-1)}{2}} P_q(n-k,k-1) e_{n-2k+1}  $$
for $n\in \mathbb{N}$,
where $\lceil x \rceil$ is the ceiling function and $P_q(n,k) = \frac{[n]_q!}{[n-k]_q!}$. Then $D$ satisfies
$[D,A]= P^{(0)}$ on $\mathcal{F}_{\mathrm{alg}}(H).$
\end{prop}
\begin{proof}
We prove this by induction of $n$. Suppose we have the formula for $n=2m-2, 2m-1$. then we compute 
\begin{align*}
 -[2m-1]_q D &e_{2m-2} + A D e_{2m-1} \\
 &=-[2m-1]_q \sum_{k=1}^{m-1} (-1)^{k-1} q^{\frac{k(k-1)}{2}} P_q(2m-k-2,k-1) e_{2m-2k-1} \\
 & \qquad + \sum_{k=1}^m (-1)^{k-1} q^{\frac{k(k-1)}{2}} P_q(2m-k-1,k-1) A e_{2m-2k} \\ 
 &= \sum_{k=1}^m (-1)^{k-1} q^{\frac{k(k-1)}{2}} \ \{-[2m-1]_q P_q(2m-k-2,k-1) \\ 
 & \qquad+ [2m-2k]_q P_q(2m-k-1,k-1) \} e_{2m-2k-1} \\
 &  \qquad+ \sum_{k=1}^m (-1)^{k-1} q^{\frac{k(k-1)}{2}} P_q(2m-k-1,k-1) e_{2m-2k +1} 
 \end{align*}
 We also have
 \begin{align*}
 & -[2m-1]_q P_q(2m-k-2,k-1) + [2m-2k]_q P_q(2m-k-1,k-1)\\ 
&\qquad\qquad= (-[2m-1]_q [2m-2k]_q + [2m-k-1]_q[2m-2k]_q)P_q(2m-k-2,k-2) \\
 &\qquad\qquad= - q^{2m-k-1} [k]_q [2m-2k]_q P_q(2m-k-2,k-2).
 \end{align*}
 Thus we have
 \begin{eqnarray*}
 & \ & -[2m-1]_q D e_{2m-2} + A D e_{2m-1} \\ 
 &=& \sum_{k=1}^{m-1} (-1)^{k} q^{\frac{k(k+1)}{2}} q^{2m-2k-1} [k]_q [2m-2k]_q P_q(2m-k-2,k-2) e_{2m-2k-1} \\
 & \ & + \sum_{k=1}^m (-1)^{k-1} q^{\frac{k(k-1)}{2}} P_q(2m-k-1,k-1) e_{2m-2k +1} \\
 &=& e_{2m-1} + \sum_{k=2}^m (-1)^{k-1} q^{\frac{k(k-1)}{2}} \{P_q(2m-k-1,k-1) \\
 &\ & + q^{2m-2k+1} [k-1]_q [2m-2k+2]_q P_q(2m-k-1,k-3)\} e_{2m-2k +1}. \\
 \end{eqnarray*}
 This is equal to $D e_{2m}$ since we have
 \begin{eqnarray*}
& \ & P_q(2m-k-1,k-1)+ q^{2m-2k+1} [k-1]_q [2m-2k+2]_q P_q(2m-k-1,k-3)\\
 &=& ([2m-2k+2]_q[2m-2k+1]_q+ q^{2m-2k+1} [k-1]_q [2m-2k+2]_q)P_q(2m-k-1,k-3)\\
 &=& [2m-2k+2]_q[2m-k]_q P_q(2m-k-1,k-3) \\
 &=& P_q(2m-k,k-1). 
 \end{eqnarray*}
 Similarly, we obtain $ D e_{2m+2} = -[2m+1]_q D e_{2m}+A D e_{2m+1},$ which implies $[D,A] e_n = P^{(0)}e_n $
\end{proof}

Since $\langle D e_n, e_0 \rangle_q $ is $(-1)^{m-1} q^{\frac{m(m-1)}{2}} [m-1]_q!$ when $n = 2m-1$, and $0$ when $n=2m$, we can formally compute $D^*e_0$ as
$$D^*e_0 = \sum_{m=1}^{\infty}(-1)^{m-1}q^{\frac{m(m-1)}{2}}\frac{[m-1]_q!}{[2m-1]_q!}e_{2m-1},$$
and the square of its Hilbert norm is
$$\|D^*e_0\|^2 = \sum_{m=1}^{\infty} |q|^{m(m-1)}\frac{([m-1]_q!)^2}{[2m-1]_q!}$$  
which is finite for $-1 < q \le 1$ by the ratio test. This implies $e_0 \in \mathrm{dom}(D^*)$. Let us collect this in the following corollary.

\begin{cor}\label{cor:xi}
For all $-1<q<1$, the vacuum vector $e_0$ is in the domain of the adjoint of the normalized dual operator and the conjugate variable for the $q$-Gaussian is given by
$$\xi=D^*e_0 = \sum_{m=1}^{\infty}(-1)^{m-1}q^{\frac{m(m-1)}{2}}\frac{[m-1]_q!}{[2m-1]_q!}e_{2m-1}.$$
\end{cor}

\begin{rem}
We remark that the polynomial corresponding to $e_n$ is the $n$th $q$-Hermite polynomial, $Q[e_n] =H_n(x|q)$. There is a relation between $q$-Hermite polynomials and the Chebyshev polynomials $U_n$ of the second kind, which is the $q=0$ version of $q$-Hermite. This can be stated as follows (for example, see Lemma 5.57 in \cite{CI20}),
$$U_n(x\sqrt{1-q}) = \sum_{k=0}^{\lfloor n/2 \rfloor} (-1)^k q^{\frac{(k+1)k}{2}} \binom{n-k}{k}_q \sqrt{1-q}^{n-2k}H_{n-2k}(x|q)$$
where 
$$\binom{n-k}{k}_q = \frac{[n-k]_q!}{[k]_q![n-2k]_q!}. $$
In particular, by computing the coefficient of $H_0(x|q) = 1$, we have 
$$\tau_q(U_{2n}(A\sqrt{1-q})) = (-1)^n q^{\frac{(n+1)n}{2}}$$ and $\tau_q(U_{2n-1}(A\sqrt{1-q}))=0$ for any $n \in \mathbb{N}.$

In the one-variable case, there is actually a general formula for the conjugate variable $\xi$ in terms of Chebyshev polynomials (provided this sum converges), namely (see Exercise 8.12 in \cite{MS17})
$$\xi = \sum_{n=1}^{\infty} \tau(U_{n-1}(X))C_n(X)$$
where $C_n(X)$ is the Chebyshev polynomial of the first kind.
If we apply this formula to $X=A\sqrt{1-q}$ and recall that a rescaling of the random variable by a factor $\alpha$ results in the rescaling of the conjugate variable by a factor $1/\alpha$, then we obtain from this general formula the following expression for the conjugate variable of the $q$-Gaussian:
$$\sqrt{1-q}\sum_{n=0}^{\infty} (-1)^{n}q^{\frac{(n+1)n}{2}} C_{2n+1}(A\sqrt{1-q}).$$
We can transform this formula into the sum of $q$-Hermite polynomials by using $C_1 (X)= U_1(X)=X$, $C_n(X) = U_n(X) - U_{n-2}(X)$ for $n \ge 2$ and the relation from above between $U_n$ and $H_n(x|q)$, resulting after some reformulations in
$$\sum_{m=0}^{\infty}(-1)^m q^{\frac{(m+1)m}{2}} (1-q)^{m+1} \sum_{n=m}^{\infty} q^{(n+1)(n-m)}(1+ q^{n+1})\binom{n+m+1}{n-m}_q e_{2m+1}.$$
By using the non-trivial identity
$$(1-q)^{m+1} \sum_{n=m}^{\infty} q^{(n+1)(n-m)}(1+ q^{n+1})\binom{n+m+1}{n-m}_q = \frac{[m]_q!}{[2m+1]_q!} $$
for all $m \in \mathbb{Z}_{\ge 0}$, we recover thus indeed our formula from Cor. \ref{cor:xi}.

\end{rem}

\section{Multi-variable case}\label{MultiVariable}
A similar reduction gives us a formula for the multi-variable case. 
Let us consider the $q$-deformed Fock space $\mathcal{F}_q(H)$ of a $d$-dimensional Hilbert space $H$ with orthonormal basis $e_1,\ldots,e_d$. Let $\{A_i\}_{i=1}^d$ be $q$-Gaussians with respect to $\{e_i\}_{i=1}^d$. 

Then the equation $[D_i, A_j] = \delta_{i j}P^{(0)}$, together with $D_i e_0=0$, allows us to determine $D_i$ inductively. 
\begin{exam}
Since we have $e_{j w}=A_j e_w - l_j e_w$ for $j \in [d]$ and $w \in [d]^*$, we have by applying $D_i$ and $[D_i, A_j] = \delta_{i j}P^{(0)}$,
$$D_i e_{j w}=D_i A_j e_w - D_i l_j e_w=A_j D_i e_w + \delta_{ij}P^{(0)}e_w-D_i l_j e_w.$$
This gives the recursion for $D_i e_w$ with $D_i e_0=0$. For example, we have for $i,j_1,j_2,j_3 \in [d]$,
\begin{eqnarray*}
D_i e_{j_1} &=& A_{j_1}D_i e_0 + \delta_{i j_1} P^{(0)} e_0 = \delta_{i j_1}e_0\\
D_i e_{j_2j_1}&=&A_{j_2}D_i e_{j_1}-\delta_{j_2j_1}D_i e_0= \delta_{i j_1} e_{j_2}\\
D_i e_{j_3 j_2 j_1} &=&A_{j_3}D_i e_{j_2 j_1}-\delta_{j_3j_2}D_i e_{j_1}-q\delta_{j_3j_1}D_i e_{j_2}\\
&=&\delta_{i j_1}A_{j_3}e_{j_2}-\delta_{j_3j_2}\delta_{i j_1}e_0-q\delta_{j_3j_1}\delta_{i j_2}e_0 \\
&=&\delta_{i j_1}e_{j_3 j_2} - q \delta_{i j_2}\delta_{j_3 j_1}e_0.
\end{eqnarray*}
Similarly, we obtain the following formulas for $j_4,j_5,j_6 \in [d]$,
\begin{eqnarray*}
D_i e_{j_4j_3j_2j_1} &=& \delta_{i j_1} e_{j_4 j_3 j_2} - q^2 \delta_{i j_2} \delta_{j_4 j_1} e_{j_3} - q \delta_{i j_2} \delta_{j_3 j_1}e_{j_4}\\
D_i e_{j_5j_4j_3j_2j_1} &=&
\delta_{i j_1}e_{j_5j_4j_3j_2}-\delta_{i j_2}(q^3\delta_{j_5j_1}e_{j_4 j_3} + q^2 \delta_{j_4j_1}e_{j_5j_3} + q \delta_{j_3j_1}e_{j_5j_4})
\\ &\ &+ \delta_{i j_3}(q^4 \delta_{j_5 j_1} \delta_{j_4 j_2} + q^3 \delta_{j_5j_2} \delta_{j_4j_1}) e_0 \\ 
D_i e_{j_6j_5j_4j_3j_2j_1}
&=& \delta_{i j_1}e_{j_6 j_5 j_4 j_3 j_2} \\ 
&\ & - \delta_{i j_2}
(q^4 \delta_{j_6j_1} e_{j_5j_4j_3} + q^3 \delta_{j_5 j_1} e_{j_6 j_4 j_3} + q^2 \delta_{j_4 j_1} e_{j_6 j_5 j_3} + q \delta_{j_3 j_1} e_{j_6j_5j_4})
\\ & \ & + \delta_{i j_3} (q^6 \delta_{j_6 j_1} \delta_{j_5 j_2} e_{j_4} +q^5 \delta_{j_6 j_1} \delta_{j_4 j_2} e_{j_5}
+ q^5 \delta_{j_6 j_2}\delta_{j_5 j_1}e_{j_4}  \\ & \ &
+q^4 \delta_{j_6 j_2} \delta_{j_4 j_1}e_{j_5}
+q^4 \delta_{j_5 j_1} \delta_{j_4 j_2}e_{j_6}
+ q^3 \delta_{j_5 j_2}\delta_{j_4 j_1}e_{j_6}).
\end{eqnarray*}
\end{exam}

\begin{rem}
As in Remark \ref{q1}, we cannot define a linear operator $D_i$ in the case $q=-1$. But now the case $q=1$ also has to be excluded if $d \ge 2$. 
For example, we have $e_1 \otimes e_2 - e_2 \otimes e_1 = 0 $ in $\mathcal{F}_1(H)$, but
$$D_1(e_1 \otimes e_2 - e_2 \otimes e_1) = -e_2 \neq 0. $$ 
For $-1 < q <1$, on the other hand, we can define $D_i$ since the operator $\bigoplus_{n=0}^{\infty} P^{\otimes n}$ is strictly positive (see Section \ref{Preliminaries}) and $\{e_w\}_{w \in[d]^*}$ forms a linear basis of $\mathcal{F}_{\mathrm{alg}}(H) \subset \mathcal{F}_q(H).$
\end{rem}
From these examples, we can guess that the general formula for $D_i e_{j_n\cdots j_1}$ is characterized by partitions of $n+1$ vertices $n>n-1>\cdots>1>i$ and counting their crossings. However, the usual definition of crossings does not work in this setting.
\begin{exam}\label{example}
In the examples above, we pick in $D_i e_{j_5j_4j_3j_2j_1}$ the term
$$q^4 \delta_{i j_3}\delta_{j_5 j_1} \delta_{j_4 j_2}.$$
This corresponds to the partition $\{(i,3),(1,5),(2,4)\}$.
Since $(1,5)$ has no crossing with $(2,4)$, the number of crossings is $2$, while the coefficient above is $q^4$.

This phenomenon also happens for the other terms
$$q^6\delta_{i j_3}\delta_{j_6 j_1} \delta_{j_5 j_2} e_{j_4}, \ q^5 \delta_{i j_3}\delta_{j_6 j_1} \delta_{j_4 j_2} e_{j_5}, \ q^4 \delta_{i j_3}\delta_{j_5 j_1} \delta_{j_4 j_2}e_{j_6}$$
in $D_i e_{j_6j_5j_4j_3j_2j_1}$.
\end{exam}

We need to change the rules of counting crossings for the precise formula of $D_i e_w$.
Here we list the rules of drawing partitions that are compatible with the formula.
We remark that a similar method of counting crossings appears in Definition 3.13 in \cite{Avsec11}.  
\begin{enumerate}
\item Consider $n+1$ vertices $n>n-1>\cdots>1>0 $.\\
\item $0$ must be connected to some $k \in \{1,\ldots,n\}$ with height $1$. \\
\item $l \in \{1,\ldots,k-1\}$ must be coupled with one of $\{k+1,\ldots,n\}$ with height $l+1$. \\
\item Vertices which are not coupled with $\{1,\ldots,k-1\}$ should be singletons and are drawn with straight lines to the top. 
\end{enumerate}
We define $B(n+1)$ as a set of partitions that satisfy the above rules.  
For $\pi \in B(n+1)$,  we denote by $p(\pi)$ the set of parings in $\pi$ and by $s(\pi)$ the singletons in $\pi$ 

\begin{exam}

Let us see what happens with the number of crossings if we follow the drawing rules from above. The term $q^4 \delta_{i j_3}\delta_{j_5 j_1} \delta_{j_4 j_2}$ from Example \ref{example} is now represented by the following crossing partition:
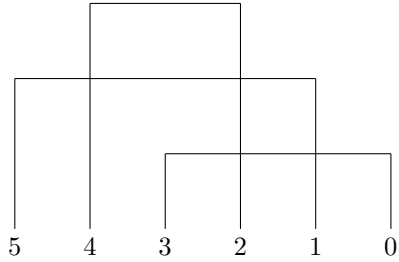
\begin{center}
    \crossing{5}{3}{0/3,1/5,2/4}{}
\end{center}
Note that the number of crossings in the picture above is now indeed $4$, corresponding to the factor $q^4$. Similarly, the factors in the contributions $q^6\delta_{i j_3}\delta_{j_6 j_1} \delta_{j_5 j_2} e_{j_4}$, $q^5 \delta_{i j_3}\delta_{j_6 j_1} \delta_{j_4 j_2} e_{j_5}$ and $q^4 \delta_{i j_3}\delta_{j_5 j_1} \delta_{j_4 j_2}e_{j_6}$ are accounted for correctly by the following partitions:
 
 \begin{center}
    \crossing{6}{3}{0/3,1/6,2/5}{4} \\
     $q^6\delta_{i j_3}\delta_{j_6 j_1} \delta_{j_5 j_2} e_{j_4}$
\end{center}
 
\begin{center}
    \crossing{6}{3}{0/3,1/6,2/4}{5} \\
    $q^5 \delta_{i j_3}\delta_{j_6 j_1} \delta_{j_4 j_2} e_{j_5}$
\end{center}
 
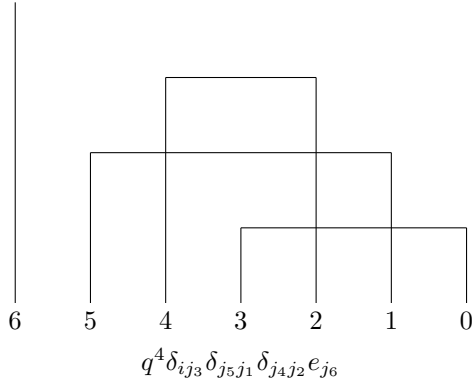
\begin{center}
    \crossing{6}{3}{0/3,1/5,2/4}{6}\\
    $q^4 \delta_{i j_3}\delta_{j_5 j_1} \delta_{j_4 j_2}e_{j_6}$
\end{center}
\end{exam}
We identify $0$ with the index of a dual system and $k \in \{1,\ldots,n\}$ with a letter $j_k \in [d]$ for a given word $j_n j_{n-1}\cdots j_1$. Then our examples from above motivate the following formula for $D_i$.

\begin{prop}\label{formula}
For $i, j_1,\ldots,j_n \in [d]$, we define densely defined unbounded operators $D_1,\ldots,D_d$, whose domains are the algebraic Fock space $\mathcal{F}(H)_{\mathrm{alg}}$, by linear extension of
$$ D_ie_0=0,\qquad D_i e_{j_n\cdots j_1} = \sum_{\pi \in B(n+1)} (-1)^{\pi(0)-1}q^{\mathrm{cross}(\pi)} \delta_{p(\pi)} e_{s(\pi)}$$
where $\mathrm{cross}(\pi)$ is the number of crossings of $\pi$ according to our drawing rules and where $\delta_{p(\pi)} = \prod_{(k,l) \in \pi} \delta_{j_k j_l}$ with $j_0=i$ and $e_{s(\pi)} = e_{j_{k_s}\cdots j_{k_1}}$ for $s(\pi) = \{k_s > \cdots > k_1\}$. 

Then we have 
$$[D_i,A_j] = \delta_{i j} P^{(0)}$$
on the domain $\mathcal{F}_{\mathrm{alg}}(H)$.
\end{prop}

\begin{proof}
We have to show that
$$[D_i,A_j]e_0 = \delta_{i j} P^{(0)}e_0=\delta_{i j}e_0
$$
and 
$$[D_i,A_j]e_{j_n\cdots j_1} = \delta_{i j} P^{(0)}e_{j_n\cdots j_1}=0$$
for all $n>0$ and $j_1,\dots,j_n\in[d]$. The first formula is easy to check, so let us concentrate on the second one. We will there rename $j$ to $j_{n+1}$ and for better legibility we will also write sometimes $[j_n\cdots j_1]$ for $e_{j_n\cdots j_1}$.

Then we can compute on one hand
\begin{eqnarray*}
& \ & A_{j_{n+1}}D_i e_{j_n\cdots j_1} \\
&=&
\sum_{\sigma \in B(n+1)} (-1)^{\sigma(0)-1}q^{\mathrm{cross}(\sigma)} \delta_{p(\sigma)} A_{j_{n+1}}e_{s(\sigma)} \\
&=& \sum_{\sigma \in B(n+1)} (-1)^{\sigma(0)-1}q^{\mathrm{cross}(\sigma)} \delta_{p(\sigma)} e_{j_{n+1}s(\sigma)}\\
&+&  \sum_{\sigma \in B(n+1)}\sum_{k=1}^{|s(\sigma)|} (-1)^{\sigma(0)-1}q^{\mathrm{cross}(\sigma)+|s(\sigma)|-k} \delta_{j_{n+1} j_{s(\sigma)_k}} \delta_{p(\sigma)} [{s(\sigma)_{|s(\sigma)|}\cdots \check{s(\sigma)}_k \cdots s(\sigma)_{1}}].  
\end{eqnarray*}

On the other hand, we have

\begin{eqnarray*}
D_i A_{j_{n+1}}e_{j_n\cdots j_1} &=& D_i e_{j_{n+1}j_n\cdots j_1} + \sum_{l=1}^n \delta_{j_{n+1} j_{l}} q^{n-l} D_i [j_n \cdots \check{j}_{l} \cdots j_1] \\
&=& D_i e_{j_{n+1}j_n\cdots j_1} \\
&+&  \sum_{l=1}^n\sum_{\pi \in B(n)}(-1)^{\pi(0)-1}  q^{\mathrm{cross}(\pi)+n-l}\delta_{j_{n+1} j_{l}}\delta_{p(\pi)}e_{s(\pi)}, 
\end{eqnarray*}
where $\check{j}_{l}$ means to omit $j_{l}$. Note that all partitions $\pi \in B(n)$ act on $n-1$ letters $j_n,\ldots,\check{j}_{l},\ldots,j_1$ in the sum above.

Let us first see that all terms in the last sum of $A_{j_{n+1}}D_i e_{j_n\cdots j_1}$ show also up as terms in $D_i A_{j_{n+1}}e_{j_n\cdots j_1}$. 

To see this, let us consider the contribution corresponding to $\sigma \in B(n+1)$ and $k \in \{1,\ldots,|s(\sigma)|\}$. Since $s(\sigma)_k$ is a singleton, we can remove it and obtain a partition $\pi \in B(n)$.
We also take $l$ so that $l = s(\sigma)_k$. Then for these $\sigma,\pi,k,l$ we have
$$ \delta_{j_{n+1}j_{l}}\delta_{p(\pi)} e_{s(\pi)}
= \delta_{j_{n+1} j_{s(\sigma)_k}}\delta_{p(\sigma)} [s(\sigma)_{|s(\sigma)|}\cdots \check{s(\sigma)_{k}} \cdots s(\sigma)_{1}]$$ 
where $\pi$ acts on $j_n,\ldots,\check{j}_{l},\ldots,j_1$. 

For example, if we take $\sigma \in B(8)$ represented by
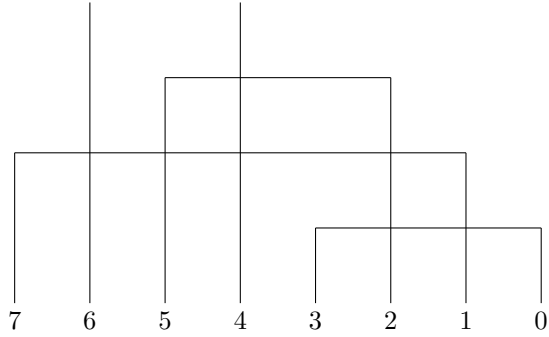
\begin{center}
    \crossing{7}{3}{0/3,1/7,2/5}{4,6}
\end{center}
 and $k = 1$, then we take $l=4$ and $\pi \in B(7)$ represented by
 
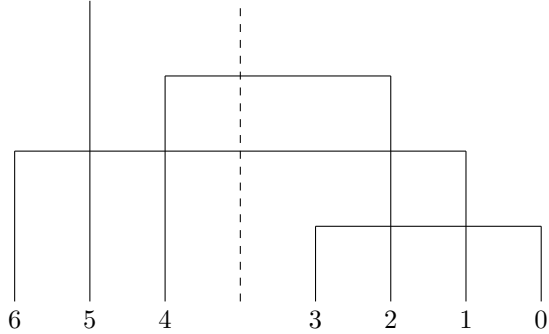
\begin{center}
 \begin{tikzpicture}
 \foreach \x in {6,5,4}
\coordinate (\x) at (6-\x,0) node [below] at (\x) {\x};
\foreach \x in {3,2,1,0}
\coordinate (\x) at (7-\x,0) node [below] at (\x) {\x};
\foreach \x/\y in {1/6,2/4}
{\draw (\x)--(7-\x,\x+1);
\draw (\y)--(6-\y, \x+1);
\draw (7-\x,\x+1)--(6-\y, \x+1);}
\draw (5)--(1, 4);
\draw (3)--(4,1);
\draw (0)--(7,1);
\draw (4,1)--(7,1);
\begin{scope}[dashed]
\draw (3,0)--(3,4);
\end{scope}[dashed]
\end{tikzpicture}
 \end{center}
where we ignore the dashed line.

Note that we have also $\pi(0) = \sigma(0)$ since we removed $s(\sigma)_k$ which is on the left of $\sigma(0)$. 
By definition, the difference between $\sigma$ and $\pi$ is only the singleton $s(\sigma)_k$ and the difference between $\mathrm{cross}(\sigma)$ and $\mathrm{cross}(\pi)$ is the number of crossing points on the line of $s(\sigma)_k$. Recall that for $\pi \in B(n+1)$ the vertices that are not singletons must be coupled. Therefore the number of crossing points on the line of $s(\sigma)_k$ is equal to the number of vertices which are on the left of $s(\sigma)_k$ and not a singleton, which is equal to $(n-l)-(|s(\sigma)|-k)$; thus we have
$$ \mathrm{cross}(\pi)+n-l=\mathrm{cross}(\sigma)+|s(\sigma)|-k.$$ 
This implies that the contribution corresponding to $\sigma$ and $k$ in the second sum of $A_{j_{n+1}}D_i e_{j_n\cdots j_1}$ shows also up as a contribution corresponding to $\pi$ and $l$ in 
$D_i A_{j_{n+1}}e_{j_n\cdots j_1}$

Next, we need to identify the remaining terms of the last sum.
Note that the $(\pi,l)$ which we can get under the above identification from $(\sigma,k)$ 
can be identified with partitions $\pi'$ of $n+2$ vertices $n+1>n>\cdots>1>0$ such that $n+1$ is coupled with some $n+1> k > \pi'(0)$ and $\pi' \setminus (n+1,k)$ belongs to $B(n)$ in an order preserving way. Thus the $l \in \{1,\ldots,n\}$ and $\pi \in B(n)$ such that $l >\pi(0)$ are exactly those terms corresponding to all possible $(\sigma,k)$. 

So we have
\begin{multline*}
-[D_i,A_{j_{n+1}}]e_{j_n\cdots j_1}
 = \sum_{\sigma \in B(n+1)} (-1)^{\sigma(0)-1}q^{\mathrm{cross}(\sigma)} \delta_{p(\sigma)} e_{j_{n+1}s(\sigma)}- D_i e_{j_{n+1}j_n\cdots j_1}\\
 + \sum_{l=1}^n\sum_{\substack{\pi \in B(n)\\ \pi(0)\ge l}}(-1)^{\pi(0)}  q^{\mathrm{cross}(\pi)+n-l}\delta_{j_{n+1} j_{l}}\delta_{p(\pi)}e_{s(\pi)}.
\end{multline*}
In order to see that this is actually equal to zero,
we need now to understand the condition $l \le \pi(0)$ in the last sum.

In this case, we can associate the term $\delta_{j_{n+1} j_{l}}\delta_{p(\pi)}e_{s(\pi)}$ to a partition $\pi'$ in $B(n+2)$; $\pi'$ is given by coupling $n+1$ with $l \le \pi(0)$ and requiring that $\pi' \setminus (n+1,l)$ is equal to $\pi$. We have then 
$$\pi(0)+1 = \pi'(0),$$ 
since $l \le \pi(0)$ is inserted into $\pi'$. 
 Moreover the difference between $\mathrm{cross}(\pi')$ and $\mathrm{cross}(\pi)$ is the number of crossing points on the pair $(n+1,l)$. Note that in our definitions, crossing points on $(n+1,l)$ consist of the coupling with $\pi'(0)>j>l$ (double count) and the coupling with $l>j\ge 0$ (single count) and also singletons of $\pi$ (single count). Therefore, this difference is equal to $n-l$ and we have $$\mathrm{cross}(\pi') = \mathrm{cross}(\pi) + n-l.$$

For example, take $\pi \in B(8)$ represented by
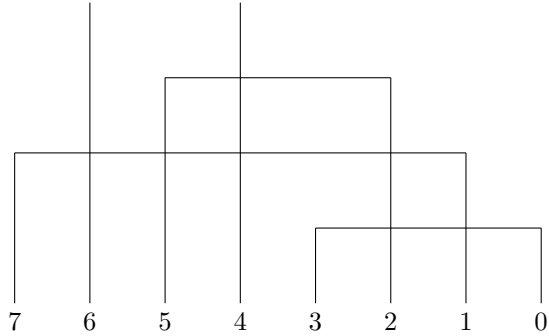
\begin{center}
    \crossing{7}{3}{0/3,1/7,2/5}{4,6}
\end{center}
and take $l= 2$. Then $\pi'\in B(10)$ is represented by
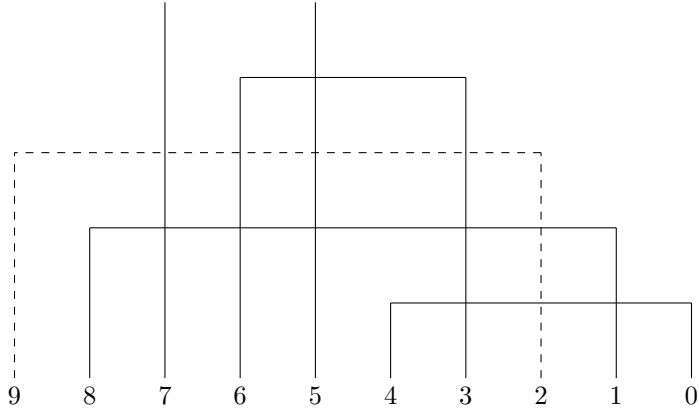
\begin{center}
 \begin{tikzpicture}
  \foreach \x in {9,8,7,6,5,4,3,2,1,0}
\coordinate (\x) at (9-\x,0) node [below] at (\x) {\x};
\foreach \x/\y in {0/4,1/8,3/6}
{\draw (\x)--(9-\x,\x+1);
\draw (\y)--(9-\y, \x+1);
\draw (9-\x,\x+1)--(9-\y, \x+1);}
\begin{scope}[dashed]
\draw (2)--(7,3);
\draw (9)--(0,3);
\draw (7,3)--(0,3);
\end{scope}[dashed]
\draw (5)--(4,5);
\draw (7)--(2,5);
 \end{tikzpicture}
 \end{center}
Here we have
$$\pi'(0)=4,\quad \pi(0)=3:\qquad 3+1=4$$
and
$$
\mathrm{cross}(\pi')=13,\quad
\mathrm{cross}(\pi)=7\quad n=8,\quad l=2: \qquad13=7+(8-2).
$$

By combining these results we obtain
\begin{align*}
-[D_i,A_{j_{n+1}}&]e_{j_n\cdots j_1}
+  D_i e_{j_{n+1}j_n\cdots j_1}\\
 &=\sum_{\sigma \in B(n+1)} (-1)^{\sigma(0)-1}q^{\mathrm{cross}(\sigma)} \delta_{p(\sigma)} e_{j_{n+1}s(\sigma)}\\
&\qquad\qquad\qquad  + \sum_{l=1}^n\sum_{\substack{\pi \in B(n)\\ \pi(0)\ge l}}(-1)^{\pi(0)}  q^{\mathrm{cross}(\pi)+n-l}\delta_{j_{n+1} j_{l}}\delta_{p(\pi)}e_{s(\pi)} \\
 &= \sum_{\sigma \in B(n+1)}(-1)^{\sigma(0)-1}q^{\mathrm{cross}(\sigma)} \delta_{p(\sigma)} e_{j_{n+1}s(\sigma)}\\
 &\qquad\qquad\qquad
  + \sum_{\substack{\pi' \in B(n+2)\\ \pi'(n+1) \mathrm{\ is \ not \ a \ singleton}}} (-1)^{\pi'(0)-1}  q^{\mathrm{cross}(\pi')}\delta_{p(\pi')}e_{s(\pi')} \\
 &=\sum_{\pi' \in B(n+2)} (-1)^{\pi'(0)-1}  q^{\mathrm{cross}(\pi')}\delta_{p(\pi')}e_{s(\pi')}\\
 &= D_i e_{j_{n+1}j_n\cdots j_1},
\end{align*}
and thus $[D_i,A_{j_{n+1}}]e_{j_n\cdots j_1}=0$, which proves our assertion.
\end{proof}
In the following we want to use this proposition to conclude that $e_0$ lies in the domain of $D_i^*$ and actually also derive a formula for $D_i^*e_0$.

\begin{thm}\label{conjugatesystem}
For any $-1<q<1$, there exists a normalized dual system and thus a conjugate system for $q$-Gaussians $A = (A_1,\ldots,A_d)$. 
\end{thm}
\begin{proof}
By Theorem \ref{shlyakhtenko}, it suffices to see that $D_i^* e_0$ exists in $\mathcal{F}_q(H)$. 
In order to see that $e_0$ is in the domain of $D_i^*$, we have to show that the linear functional $\langle D_i \ \cdot \ , e_0 \rangle_q$ is bounded on the algebraic Fock space.

Let us take $\sum_{w \in [d]^*} \alpha_w e_w \in \mathcal{F}_{\mathrm{alg}}(H)$. Note that we can compute $\langle D_i e_w , e_0 \rangle_q$ by counting summands without singletons by Proposition \ref{formula} and in that case the word length $|w|$ must be odd and $\pi \in B(2m)$ must connect $0$ to $m$ for $|w|=2m-1$. Thus we have
\begin{eqnarray*}
\langle D_i \sum_{w \in [d]^*} \alpha_w e_w , e_0 \rangle_q &=& \sum_{m=1}^{\infty} \sum_{|w|=2m-1} \sum_{\substack{\pi \in B(2m) \\ \pi(0)=m}}\alpha_w (-1)^{m-1} q^{\mathrm{cross}(\pi)} \delta_{p(\pi)w} \\
&=& \sum_{m=1}^{\infty} \sum_{\substack{\pi \in B(2m) \\ \pi(0)=m}} \sum_{|w|=2m-1} \alpha_w (-1)^{m-1} q^{\mathrm{cross}(\pi)} \delta_{p(\pi)w}.
\end{eqnarray*}
Note that we write now $\delta_{p(\pi)w}$ for
$\delta_{p(\pi)} = \prod_{(k,l) \in \pi} \delta_{j_k j_l}$ in order to make the dependency on $w={j_{2m-1}\cdots j_1}$ explicit. 
For each $\pi \in B(2m)$, let us consider words $w$ with $|w| = 2m-1$ such that $\delta_{p(\pi)w} = 1$. Such words can be represented by $\pi_p(w) i  w$ where $w$ is any word with $|w| = m-1$ and $\pi_p \in S_{m-1}$ is a permutation such that $\delta_{p(\pi)\pi_p(w) i  w} =1$. Note that there is a one to one correspondence between $\pi$ and $\pi_p$. For example, the partition $\pi$ in $B(8)$ represented by
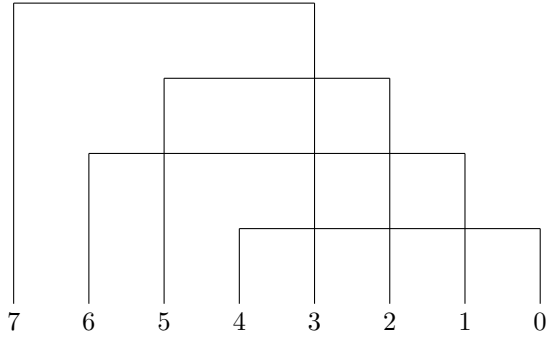
\begin{center}
    \crossing{7}{4}{0/4,1/6,2/5,3/7}{}
\end{center}
induces the element $\pi_p$ in $S_3$ represented by 
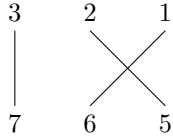
\begin{center}

  \begin{tikzpicture}
  \foreach \x in {3,2,1}
  \coordinate (\x) at (3-\x,1) node [above] at (\x) {\x};
  \foreach \x in {7,6,5}
  \coordinate (\x) at (7-\x,0) node [below] at (\x) {\x};
  \draw (3)--(7);
  \draw (2)--(5);
  \draw (1)--(6);
  \end{tikzpicture}
  
\end{center}
The important observation is that we have 
$$\mathrm{cross}(\pi) = \frac{m(m-1)}{2} + |\pi_p| $$
where $|\pi_p|$ is the number of inversions of $\pi_p$. 
Actually, when we take, for $1 \le k \le m-1$, the pair $(k,\pi(k)) \in \pi$ , then this pair crosses with $k$ pairs $(0,m)$, $(1,\pi(1))$, $\ldots$, $(k-1,\pi(k-1))$ on the right area $(m> m-1 > \cdots > 0)$, which implies the number of crossings in the right area is $\sum_{k=1}^{m-1} k = \frac{m (m-1)}{2}$. In addition, $(k,\pi(k)),(l,\pi(l))$ ($k < l \in \{1,\ldots,m-1\})$ are crossing in the left area ($2m-1> 2m-2 > \cdots > m$) if and only if $\pi(l) < \pi(k)$, which implies the number of crossings in the left area is $|\pi_p|$.     

Thus we can continue our calculation as follows:
\begin{align*}
\langle D_i \sum_{w \in [d]^*} \alpha_w e_w , e_0 \rangle_q 
&= \sum_{m=1}^{\infty} \sum_{\substack{\pi \in B(2m) \\ \pi(0)=m}} \sum_{|w|=2m-1} \alpha_w (-1)^{m-1} q^{\mathrm{cross}(\pi)} \delta_{p(\pi)w} 
\\ &= \sum_{m=1}^{\infty} (-1)^{m-1}q^{\frac{m(m-1)}{2}}\sum_{\pi_p \in S_{m-1}} \sum_{|w|=m-1} \alpha_{\pi_p(w)iw}  q^{|\pi_p|} \\ 
&= \sum_{m=1}^{\infty} (-1)^{m-1}q^{\frac{m(m-1)}{2}}\sum_{|w|=m-1}\sum_{\pi_p \in S_{m-1}} \sum_{|v| = m-1}\delta_{v \pi_p(w)} \alpha_{viw} q^{|\pi_p|}\\
&= \sum_{m=1}^{\infty} (-1)^{m-1}q^{\frac{m(m-1)}{2}}\sum_{|w|=m-1}\sum_{|v| = m-1}\sum_{\pi_p \in S_{m-1}}  \langle  e_v,e_{\pi_p(w)}\rangle \alpha_{viw} q^{|\pi_p|}\\
&= \sum_{m=1}^{\infty} (-1)^{m-1}q^{\frac{m(m-1)}{2}}\sum_{|w|=m-1}\langle \sum_{|v| = m-1} \alpha_{viw}e_v ,e_w \rangle_q.
\end{align*}
By the triangle inequality, we have
\begin{eqnarray*}
|\langle D_i \sum_{w \in [d]^*} \alpha_w e_w , e_0 \rangle_q|
& \le & \sum_{m=1}^{\infty} |q|^{\frac{m(m-1)}{2}}\sum_{|w|=m-1}|\langle \sum_{|v| = m-1} \alpha_{viw}e_v ,e_w \rangle_q|\\
& \le & \sum_{m=1}^{\infty} |q|^{\frac{m(m-1)}{2}}\sum_{|w|=m-1}\| \sum_{|v| = m-1} \alpha_{viw}e_v\|_q\cdot \|e_w \|_q.
\end{eqnarray*}
Note that $\|e_w \|^2_q \le \sum_{\pi \in S_{m-1}} |q|^{|\pi|} = [m-1]_{|q|} !$.

On the other hand, we can write 
$$\sum_{|v| = m-1} \alpha_{viw}e_v=r_{iw}\sum_{|v| = m-1} \alpha_{viw} e_{viw}
=r_{iw}\sum_{|v| = 2m-1} \alpha_{v} e_{v}
$$ 
where 
$r_{w} = r_{w_1} \cdots r_{w_n}$ for $w = w_1 \cdots w_n$
is the free right annihilation operator of the word $w$.

By Lemma \ref{FreeRightAnnihilation}, the free right annihilation operators $r_1,\ldots,r_d$ are bounded and their operator norms are less than $C= \sqrt{w(q)}^{-1}$ where $w(q)$ is a positive constant which appears in \cite{Bozejko98}. Since $r_{iw}$ is in our case a product of $m$ such free right annihilation operators, we have $\Vert r_{iw}\Vert\leq C^m$ and thus
$$\| \sum_{|v| = m-1} \alpha_{viw}e_v\|_q \le C^m \|\sum_{|v| = 2m-1} \alpha_{v}e_v \|_q \le C^m \|\sum_{w \in [d]^*} \alpha_w e_w\|_q.$$

So, finally, we have the following estimate:
\begin{eqnarray*}
|\langle D_i \sum_{w \in [d]^*} \alpha_w e_w , e_0 \rangle_q|
&\le& \|\sum_{w \in [d]^*} \alpha_w e_w\|_q \sum_{m=1}^{\infty} |q|^{\frac{m(m-1)}{2}}\sum_{|w|=m-1} C^m \sqrt{[m-1]_{|q|}!} \\
&=& \|\sum_{w \in [d]^*} \alpha_w e_w\|_q \sum_{m=1}^{\infty} |q|^{\frac{m(m-1)}{2}} d^{m-1} C^m \sqrt{[m-1]_{|q|}!}
\end{eqnarray*}
and, by the ratio test, we can check that $$\sum_{m=1}^{\infty} |q|^{\frac{m(m-1)}{2}} d^{m-1} C^m \sqrt{[m-1]_{|q|}!} < \infty.$$ 
This implies that the linear functional $\langle D_i \ \cdot \ , e_0 \rangle_q$ is bounded and therefore $e_0 \in \mathrm{dom}(D_i^*)$.
\end{proof}

\begin{cor}\label{ConcreteFormula}
Let $(D_1,\ldots,D_d)$ be the normalized dual system of the $q$-Gaussian operators, as defined in Proposition \ref{formula}. Then the corresponding conjugate system $(\xi_1,\dots,\xi_d)$ is given by 
$$\xi_i=D_i^* e_0 = \sum_{w \in [d]^*} (-1)^{|w|} q^{\frac{(|w|+1)|w|}{2}} r_{i w}^* e_w $$
where $r_{i w} = r_i r_{w_1} \cdots r_{w_n}$ for $w = w_1 \cdots w_n$.
Moreover, the series for $\xi_i$ is not only convergent with respect to the Hilbert space norm $\Vert\cdot\Vert_q$, but also with respect to the operator norm $\Vert\cdot\Vert$, if we identify  operators in $W^*(A)$ with elements in the Fock space. Thus $\xi_i=X_i e_0$, where $X_i$ is contained in the norm closure of non-commutative polynomials  $\mathbb{C}\langle A \rangle$; i.e., in particular $X_i\in W^*(A)$.
\end{cor}

\begin{proof}

In the proof of Theorem \ref{conjugatesystem}, we have seen that

\begin{eqnarray*}
\langle D_i \sum_{v \in [d]^*} \alpha_v e_v , e_0 \rangle_q &=&\sum_{m=1}^{\infty} (-1)^{m-1}q^{\frac{m(m-1)}{2}}\sum_{|w|=m-1}\langle \sum_{|v| = 2m-1} r_{iw} \alpha_{v} e_{v} ,e_w \rangle_q \\
&=&\sum_{m=1}^{\infty} (-1)^{m-1}q^{\frac{m(m-1)}{2}}\langle \sum_{|v| = 2m-1} \alpha_{v}e_v ,\sum_{|w|=m-1}r_{i w}^*e_w \rangle_q.
\end{eqnarray*}
Note that $$ \langle \sum_{|v| = 2m-1} \alpha_{v}e_v ,\sum_{|w|=m'-1}r_{i w}^*e_w \rangle_q = 0,\qquad \text{if $m \neq m'$}$$ 
since $r_{i w}^*$ maps $e_w$ (with $|w| = m'-1$) to the subspace spanned by $\{e_v\}_{|v|=2m'-1}$. This also implies 
$$\langle \sum_{|v| = 2m} \alpha_{v}e_v ,\sum_{|w|=m'-1}r_{i w}^*e_w \rangle_q = 0,\qquad \text{for any $m,m'$}.$$
Thus we obtain
\begin{multline*}
\sum_{m=1}^{\infty} (-1)^{m-1}q^{\frac{m(m-1)}{2}}\langle \sum_{|v| = 2m-1} \alpha_{v}e_v ,\sum_{|w|=m-1}r_{i w}^*e_w \rangle_q \\ 
=\langle \sum_{v \in [d]^*} \alpha_{v}e_v \ , \sum_{m=1}^{\infty} \sum_{|w|=m-1}(-1)^{m-1}q^{\frac{m(m-1)}{2}}r_{i w}^*e_w \rangle_q. 
\end{multline*}
For the operator norm, we can estimate by the triangle inequality
$$\|D^*_i e_0 \| \le \sum_{m=0}^{\infty} \sum_{|w|=m}|q|^{\frac{m(m+1)}{2}}\|r_{i w}^*e_w\|. $$
Now, we use Bożejko's Haagerup type inequality \cite{Bozejko99}, which tells us for $|w| = m $
$$ \|r_{i w}^*e_w\| \le (2m + 2)C_{|q|}^{\frac{3}{2}} \|r_{i w}^*e_w\|_q, $$
where $C_q^{-1} = \prod_{m=1}^{\infty} (1-q^m).$ Since $\|r_{i w}^*\| \le \sqrt{w(q)}^{-(m+1)}$ and $ \|e_w\|_q \le \sqrt{[m]_{|q|}!}$ (see the proof of Theorem \ref{conjugatesystem}), we have

\begin{eqnarray*}
\sum_{m=0}^{\infty} \sum_{|w|=m}|q|^{\frac{m(m+1)}{2}}\|r_{i w}^*e_w\|
& \le & \sum_{m=0}^{\infty} \sum_{|w|=m}|q|^{\frac{m(m+1)}{2}}(2m + 2)C_{|q|}^{\frac{3}{2}} \sqrt{w(q)}^{-(m+1)} \sqrt{[m]_{|q|}!}\\
& =& \sum_{m=0}^{\infty} d^m |q|^{\frac{m(m+1)}{2}}(2m + 2)C_{|q|}^{\frac{3}{2}} \sqrt{w(q)}^{-(m+1)} \sqrt{[m]_{|q|}!},
\end{eqnarray*}
which is finite by the ratio test. Since $e_w$ can be represented by non-commutative polynomials over $q$-Gaussians, $D^*_i e_0$ belongs to the norm closure of $\mathbb{C}\langle A \rangle$. 
\end{proof}

\begin{rem}
We can recover the results in Section \ref{OneVariable} for the case of one variable from Proposition \ref{formula} and Corollary \ref{ConcreteFormula} by considering $d=1$ and identifying $e_{1^m}$ with $e_m$. In particular, we have then

$$r_1^* e_m = \frac{1}{[m+1]_q} e_{m+1},$$
and
$$(r_1^*)^{m+1} e_m = \frac{1}{[m+1]_q[m+2]_q\cdots [2m+1]_q} e_{2m+1} = \frac{[m]_q!}{[2m+1]_q!}e_{2m+1}$$
and
$$\xi_1= \sum_{m=0}^\infty (-1)^{m} q^{\frac{(m+1)m}{2}} (r_1^*)^{m+1} e_m 
= \sum_{m=0}^\infty (-1)^{m} q^{\frac{(m+1)m}{2}}
\frac{[m]_q!}{[2m+1]_q!}e_{2m+1},
$$
which recovers, by replacing $m$ by $m-1$, the formula for $\xi=D^*e_0$ in Corollary \ref{cor:xi}.
\end{rem}

\section{Lipschitz conjugate}\label{LipschitzConjugate}

Let us check that the conjugate system $(\xi_1,\dots,\xi_d)$ for the $q$-Gaussian variables are Lipschitz conjugate variables, namely, for each $i\in[d]$, $\xi_i=D_i^* e_0 \in \mathrm{dom}(\overline{\partial_j})$ and $\overline{\partial_j} \xi_i \in W^*(A) \overline{\otimes} W^*(A)$ for each $j \in [d]$. For this we need to know $\partial_j e_w$ and this has a similar combinatorial formula to that of a normalized dual system. Again we have to consider a special set of partitions, consisting just of singletons and pairs, and draw them in a specific way to count their crossings.
\begin{enumerate}
    \item Consider $n+1$ vertices $n > \cdots >1>0 . $ \\
    \item The vertex $0$ must be coupled with some $k \in \{1,\ldots,n\}$ with height $1$.\\
    \item Each $l \in \{1,\ldots,k-1\}$ is a singleton or coupled with one of $\{k+1,\ldots,n\}$ with height $l+1. $\\
    \item Vertices which are not coupled with one of $\{1,\ldots,k-1\}$ should be singletons and are drawn with straight lines to the top.
\end{enumerate}
Let $C(n+1)$ be the set of partitions defined by the rules above. For each $\pi \in C(n+1)$, we define $s_l(\pi)$ and $s_r(\pi)$ as the set of singletons in the left area $n \ge k > \pi(0) $ and in the right area $\pi(0) > k \ge 1$, respectively. As before, we use the notation $\mathrm{cross}(\pi)$ for the number of crossings in the drawing according to these rules.   

For each $w \in [d]^*$, we identify $e_w$ with the noncommutative polynomial $Q[w]$ over $q$-Gaussians. We give now the combinatorial formula for $\partial_i e_{j_n \cdots j_1}$ identifying each index $j_k$ with a vertex of $k$ (where we put $j_0=i$). 
\begin{prop}\label{ncderivative}
For each $i\in[d]$, $n\in\mathbb{N}$ and  $j_1,\ldots,j_n \in [d]$, we have
$$\partial_i e_{j_n \cdots j_1} = \sum_{\pi \in C(n+1)} (-1)^{|p(\pi)|-1} q^{\mathrm{cross}(\pi)-|s_r(\pi)|} \delta_{p(\pi)} e_{s_l(\pi)} \otimes e_{s_r(\pi)}$$
As before, we denote by $p(\pi)$ the set of parings in $\pi$ and the factor $\delta_{p(\pi)}$ ensures
that $\pi$ has to pair the same indices.
\end{prop}

\begin{proof}
We will prove the formula by induction over $n$. For $n=1$, it says that $\partial_i e_j=\delta_{ij} e_0\otimes e_0$, which is clearly true. Assume now that the formula is true for $n\geq1$ and let us show it for $n+1$.
By the definition of $q$-Gaussians, we have the recursion for noncommutative polynomials $e_w$ 
$$e_{j_{n+1}\cdots j_1} = A_{j_{n+1}} e_{j_n \cdots j_1} - \sum_{k=1}^n \delta_{j_{n+1}  j_k} q^{n-k} [j_n \cdots \check{j}_k \cdots j_1], $$
which induces
\begin{multline*}
\partial_i e_{j_{n+1}\cdots j_1} = \delta_{i j_{n+1}} 1 \otimes e_{j_n \cdots j_1}+ (A_{j_{n+1}} \otimes 1) \cdot \partial_i e_{j_n \cdots j_1}\\ - \sum_{k=1}^n \delta_{j_{n+1}  j_k} q^{n-k} \partial_i [j_n \cdots \check{j}_k \cdots j_1]. 
\end{multline*}
By using the induction assumption, we can compute 
\begin{align*}
&(A_{j_{n+1}} \otimes 1) \cdot \partial_i e_{j_n \cdots j_1}=
\sum_{\pi \in C(n+1)} (-1)^{|p(\pi)|-1} q^{\mathrm{cross}(\pi)-|s_r(\pi)|} \delta_{p(\pi)} A_{j_{n+1}}e_{s_l(\pi)} \otimes e_{s_r(\pi)} \\
&= \sum_{\pi \in C(n+1)} (-1)^{|p(\pi)|-1} q^{\mathrm{cross}(\pi)-|s_r(\pi)|} \delta_{p(\pi)} e_{j_{n+1} s_l(\pi)} \otimes e_{s_r(\pi)}\\
&\quad + \sum_{\pi \in C(n+1)} (-1)^{|p(\pi)|-1}\times\\
&\qquad \times
\Bigl\{\sum_{m =1}^{|s_l(\pi)|} q^{\mathrm{cross}(\pi)-|s_r(\pi)|+|s_l(\pi)|-m}
 \delta_{p(\pi)} \delta_{j_{n+1} j_{s_l(\pi)_m}} [s_l(\pi)_{|s_l(\pi)|} \cdots \check{s_l(\pi)_m} \cdots s_l(\pi)_1] \otimes e_{s_r(\pi)} \Bigr\}.
\end{align*}
On the other hand, we can compute $\sum_{k=1}^n \delta_{j_{n+1} j_k} q^{n-k} \partial_i [j_n \cdots \check{j}_k \cdots j_1]$ as
\begin{eqnarray*}
\sum_{k=1}^n \sum_{\sigma \in C(n)} \delta_{j_{n+1} j_{k}}  (-1)^{|p(\sigma)|-1} q^{\mathrm{cross}(\sigma)-|s_r(\sigma)|+n-k} \delta_{p(\sigma)} e_{s_l(\sigma)} \otimes e_{s_r(\sigma)}
\end{eqnarray*}
where $\sigma$ acts on the word $j_n \cdots \check{j}_k \cdots j_1.$ By the same argument as in Proposition \ref{formula}, we can see that the last sum of $(A_{j_{n+1}} \otimes 1) \cdot \partial_i e_{j_n \cdots j_1}$ is canceled by $-\sum_{k=1}^n \delta_{j_{n+1} j_k} q^{n-k} \partial_i [j_n \cdots \check{j}_k \cdots j_1].$
Indeed, for each $\pi \in C(n+1)$ and $m \in \{1,\ldots,|s_l(\pi)|\}$, we take $k = s_l(\pi)_m$ and $\sigma = \pi \setminus s_l(\pi)_m \in C(n).$ Then $|p(\pi)| = |p(\sigma)|$. By counting the crossing points on the line $s_l(\pi)_m$, we have 
$$\mathrm{cross}(\pi) - \mathrm{cross}(\sigma) = n - k- (|s_l(\pi)|-m)  $$
Since we have $|s_r(\pi)| = |s_r(\sigma)|$, we have
$$\mathrm{cross}(\pi)-|s_r(\pi)|+|s_l(\pi)|-m = \mathrm{cross}(\sigma)-|s_r(\sigma)|+n-k $$
For example, we take $\pi \in C(7)$ represented by
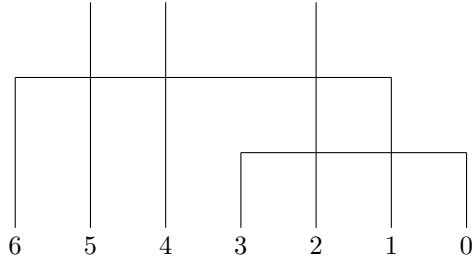
\begin{center}
 \begin{tikzpicture}
 \foreach \x in {6,5,4,3,2,1,0}
 \coordinate (\x) at (6-\x,0) node [below] at (\x) {\x};
 \draw (0)--(6,1);
 \draw (3)--(3,1);
 \draw (6,1)--(3,1);
 \draw (6)--(0,2);
 \draw (1)--(5,2);
 \draw (0,2)--(5,2);
 \draw (4)--(2,3);
 \draw (5)--(1,3);
 \draw (2)--(4,3);
 \end{tikzpicture}
 \end{center}
and take $m =1$. Then, we take $k = 4$ and $\sigma \in C(6)$ represented by 
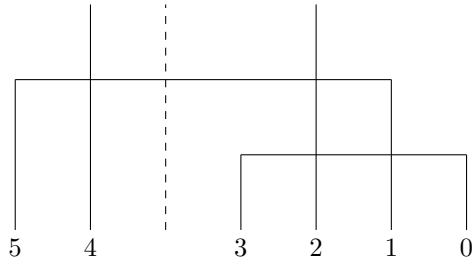
\begin{center}
 \begin{tikzpicture}
 \foreach \x in {5,4}
 \coordinate (\x) at (5-\x,0) node [below] at (\x) {\x};
 \foreach \x in {3,2,1,0}
 \coordinate (\x) at (6-\x,0) node [below] at (\x) {\x};
 \draw (0)--(6,1);
 \draw (3)--(3,1);
 \draw (6,1)--(3,1);
 \draw (5)--(0,2);
 \draw (1)--(5,2);
 \draw (0,2)--(5,2);
 \draw (4)--(1,3);
 \draw (2)--(4,3);
 \begin{scope}[dashed]
 \draw (2,0)--(2,3);
 \end{scope}[dashed]
 \end{tikzpicture}
 \end{center}
In this case, we have 
$$
|p(\pi)|=|p(\sigma)| = 2, \qquad  |s_l(\pi)|=  2, \qquad |s_r(\pi)|=|s_r(\sigma)| = 1$$
and thus
$$
\mathrm{cross}(\pi)-|s_r(\pi)|+|s_l(\pi)|-m =5-1+2-1 = 5= 4-1+6-4=\mathrm{cross}(\sigma)- |s_r(\sigma)| + n-k . 
$$
The remaining terms in $-\sum_{k=1}^n \delta_{j_{n+1} j_k} q^{n-k} \partial_i [j_n \cdots \check{j}_k \cdots j_1]$, which are characterized by $\sigma \in C(n)$ and $k \le \sigma(0)$, are corresponding to partitions $\sigma' \in C(n+2)$ which connect $n+1$ to $k$ and satisfy $\sigma' \setminus (n+1,k) = \sigma$. Then $|p(\sigma')| = |p(\sigma)| + 1 $ and by counting the crossing points on $(n+1,k)$ we have
$$\mathrm{cross}(\sigma') - \mathrm{cross}(\sigma)= n-k.  $$
We also have $|s_r(\sigma')|=|s_r(\sigma)|$, and thus
$$\mathrm{cross}(\sigma') -|s_r(\sigma')| = \mathrm{cross}(\sigma) - |s_r(\sigma)| + n -k. $$
For example, consider $\sigma \in C(7)$ represented by
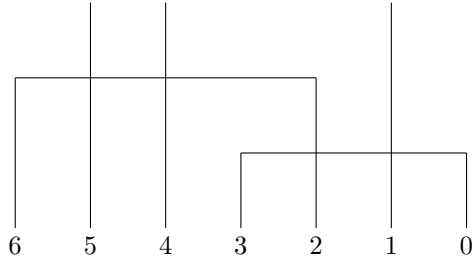
\begin{center}
 \begin{tikzpicture}

 \foreach \x in {6,5,4,3,2,1,0}
 \coordinate (\x) at (6-\x,0) node [below] at (\x) {\x};
 \draw (0)--(6,1);
 \draw (3)--(3,1);
 \draw (6,1)--(3,1);
 \draw (6)--(0,2);
 \draw (2)--(4,2);
 \draw (0,2)--(4,2);
 \draw (4)--(2,3);
 \draw (5)--(1,3);
 \draw (1)--(5,3);
 \end{tikzpicture}
 \end{center}
and take $k=2$. Then we obtain $\sigma' \in C(9)$ represented by
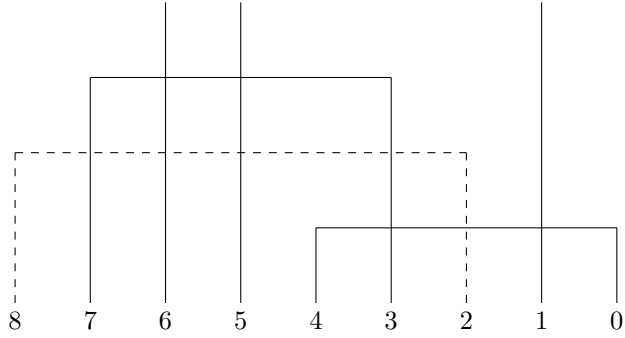
\begin{center}
 \begin{tikzpicture}
 \foreach \x in {8,7,6,5,4,3,2,1,0}
 \coordinate (\x) at (8-\x,0) node [below] at (\x) {\x};
 \draw (0)--(8,1);
 \draw (7)--(1,3);
 \draw (6)--(2,4);
 \draw (5)--(3,4);
 \draw (4)--(4,1);
 \draw (3)--(5,3);
 \draw (1)--(7,4);
 \draw (1,3)--(5,3);
 \draw (8,1)--(4,1);
 \begin{scope}[dashed]
 \draw (8)--(0,2);
 \draw (2)--(6,2);
 \draw (0,2)--(6,2);
 \end{scope}[dashed]
 \end{tikzpicture}
 \end{center}
In this case, 
$$|p(\sigma')|=3=|p(\sigma)|+1,\qquad \mathrm{cross}(\sigma')- \mathrm{cross}(\sigma) =9-4= 5=7-2.$$
The term $\delta_{i j_{n+1}} 1 \otimes e_{j_n \cdots j_1}$ is given by the partitions in $C(n+2)$ which connect $0$ to $n+1$. The sum $$\sum_{\pi \in C(n+1)} (-1)^{|p(\pi)|-1} q^{\mathrm{cross}(\pi)-|s_r(\pi)|} \delta_{p(\pi)} e_{j_{n+1} s_l(\pi)} \otimes e_{s_r(\pi)}$$ 
is given by the partitions in $C(n+2)$ such that $n+1$ is a singleton. Therefore $\partial_i e_{j_{n+1}\cdots j_1}$ is given by the partitions in $C(n+2)$ and we have proved the claimed formula for $n+1$.
\end{proof}

\begin{exam}
For $n=3$, the proposition tells us 
$$\partial_i e_{j_3 j_2 j_1}= \delta_{i j_3} 1 \otimes e_{j_2 j_1} + \delta_{i j_2} e_{j_3} \otimes e_{j_1} + \delta_{i j_1} e_{j_3 j_2} \otimes 1 - q \delta_{i j_2}\delta_{j_3 j_1}  1 \otimes 1. $$
The following four partitions characterize each term.  
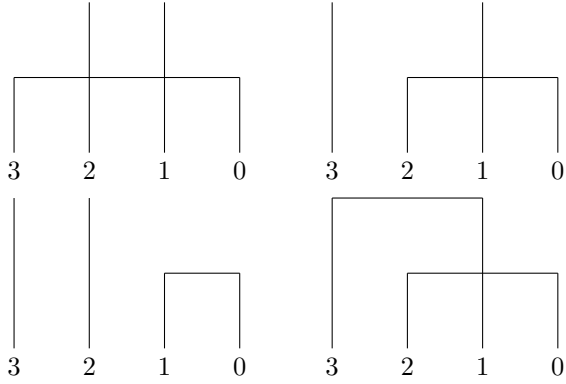
\begin{center}
\begin{tabular}{cc}
 \begin{tikzpicture}
 \foreach \x in {3,2,1,0}
 \coordinate (\x) at (3-\x,0) node [below] at (\x) {\x};
 \draw (0)--(3,1);
 \draw (3)--(0,1);
 \draw (3,1)--(0,1);
 \draw (2)--(1,2);
 \draw (1)--(2,2);
 \end{tikzpicture}
 \quad
&
\quad

 \begin{tikzpicture}
 \foreach \x in {3,2,1,0}
 \coordinate (\x) at (3-\x,0) node [below] at (\x) {\x};
 \draw (0)--(3,1);
 \draw (2)--(1,1);
 \draw (3,1)--(1,1);
 \draw (3)--(0,2);
 \draw (1)--(2,2);
 \end{tikzpicture}
 \\
 \begin{tikzpicture}
 \foreach \x in {3,2,1,0}
 \coordinate (\x) at (3-\x,0) node [below] at (\x) {\x};
 \draw (0)--(3,1);
 \draw (1)--(2,1);
 \draw (3,1)--(2,1);
 \draw (3)--(0,2);
 \draw (2)--(1,2);
 \end{tikzpicture}
 \quad
&
 \quad
 \begin{tikzpicture}
 \foreach \x in {3,2,1,0}
 \coordinate (\x) at (3-\x,0) node [below] at (\x) {\x};
 \draw (0)--(3,1);
 \draw (2)--(1,1);
 \draw (3,1)--(1,1);
 \draw (3)--(0,2);
 \draw (1)--(2,2);
 \draw (0,2)--(2,2);
 \end{tikzpicture}
 
 \end{tabular}
\end{center}

\end{exam}
\begin{cor}\label{Lipschitz_for_q_Gaussian}
The conjugate system of $q$-Gaussians is Lipschitz conjugate for $-1< q <1$.
\end{cor}

\begin{proof}
Let $i,j \in [d]$. 
We will check that the following sum
$$\partial_j \xi_i=\sum_{w \in [d]^*} (-1)^{|w|} q^{\frac{(|w|+1)|w|}{2}} \partial_j  r_{i w}^* e_w $$
converges in the operator norm of $B(\mathcal{F}_q(H)^{\otimes 2})$. Since $r_{iw}^* e_w$ is a linear span of $\{e_v\}_{|v| = 2m+1}$ for each $|w| = m$, we can write $r_{iw}^* e_w = \sum_{|v| = 2m+1} \alpha_v e_v $. (Note that $\alpha_v$ depends on $w$, but since the following estimates do not depend on $w$ we will suppress this in the notation.) Note that $\alpha_v e_0 = r_v r_{iw}^* e_w $ and we can estimate $$|\alpha_v| = \| r_v r_{iw}^* e_w\|_q \le \sqrt{w(q)}^{-3m-2} \sqrt{[m]_{|q|}!}.$$
Therefore we can estimate, by the triangle inequality,
$$ \| \partial_j r_{iw}^* e_w  \| \le \sum_{|v| = 2m+1} \sqrt{w(q)}^{-3m-2} \sqrt{[m]_{|q|}!} \ \| \partial_j e_v \|.$$
By Proposition \ref{ncderivative} and Bożejko's Haagerup type inequality \cite{Bozejko99}, we obtain
\begin{eqnarray*}
\| \partial_j e_v \| \le \sum_{\pi \in C(2m+2)} \| e_{s_l(\pi)} \|\cdot \| e_{s_r(\pi)} \| 
\le C_{|q|}^3(2m+1)^2[2m]_{|q|}!  \ |C(2m+2)| 
\end{eqnarray*}
where $|C(2m+2)|$ is the cardinality of $C(2m+2)$. Since we can regard $C(2m+2)$ as a subset of the symmetric group of degree $2m+2$, we obtain $|C(2m+2)| \le (2m+2)!$.
Therefore we have
\begin{align*}
 \sum_{w \in [d]^*}\| &(-1)^{|w|} q^{\frac{(|w|+1)|w|}{2}} \partial_j  r_{i w}^* e_w \| 
\le \sum_{m=0}^{\infty} \sum_{|w|=m} |q|^{\frac{(m+1)m}{2}} \| \partial_j  r_{i w}^* e_w \| \\
&\le \sum_{m=0}^{\infty} \sum_{|w|=m} |q|^{\frac{(m+1)m}{2}} \sum_{|v| = 2m+1} \sqrt{w(q)}^{-3m-2} \sqrt{[m]_{|q|}!} \ \| \partial_j e_v \| \\
&\le C_{|q|}^3 \sum_{m=0}^{\infty} \sum_{|w|=m} |q|^{\frac{(m+1)m}{2}} \times\\
&\qquad\times \sum_{|v| = 2m+1}\sqrt{w(q)}^{-3m-2} \sqrt{[m]_{|q|}!} (2m+1)^2[2m]_{|q|}!  (2m+2)! \\
&= C \sum_{m=0}^{\infty} |q|^{\frac{(m+1)m}{2}} (2m+1)^2 (2m+2)! \left(\frac{d}{\sqrt{w(q)}}\right)^{3m} \sqrt{[m]_{|q|}!} \ [2m]_{|q|}! .
\end{align*}
where 
$$C=(d C_{|q|}^3) /{w(q)}$$ 
is a constant which is independent of $m$ and the last sum is convergent by the ratio test. This implies $ \xi_i \in \mathrm{dom}(\overline{\partial_j})$ and $\overline{\partial_j} \xi_i \in W^*(A) \overline{\otimes} W^*(A).$
\end{proof}


\begin{rem}\label{qijcase}
It is likely that one can extend our results to more general deformations of Gaussian algebras; in particular, the case of mixed $q_{ij}$-Gaussians is quite straightforward. In this case the $q$-commutation relations
$ a_i a_j^* - qa_j^* a_i = \delta_{i j}$ are replaced by
$ a_i a_j^* - q_{ij}a_j^* a_i = \delta_{i j}$, 
where the parameters $(q_{ij})_{1 \le i,j \le d}$ just have to satisfy  $-1\leq q_{ij}=q_{ji}\leq 1$ and we are still looking on the von Neumann algebra generated by all $A_i:=a_i+a_i^*$. As for the $q$-case there exists a representation of these operators as creation or annihilation operators on the Fock space
\cite{Sp93,BS94,JSW}. The main difference is that in all formulas the factor $q$ for a crossing has to be replaced by $q_{ij}$, where $i$ and $j$ are the indices of the two crossing strings (where one should note that we only get a non-vanishing contribution from a pairing if it pairs the same indices). To be more precise, the left creation operators are defined in the same way as in Section \ref{Preliminaries}, but the left annihilation operators are defined by
$$l_i e_{j_n j_{n-1}\cdots j_1} = \sum_{k=1}^n \delta_{i j_{k}} q_{i j_n}q_{i j_{n-1}} \cdots q_{i j_{k+1}} e_{j_n \cdots \check{j_k} \cdots j_1}.$$

This induces the same combinatorial structure for $D_i e_{j_n j_{n-1}\cdots j_1}$ as in Proposition \ref{formula} as well as for $\partial_i e_{j_n j_{n-1}\cdots j_1}$ as in Proposition \ref{ncderivative}, if one replaces $q$ by the appropriate $q_{ij}$
for crossings according to our drawings in Section \ref{MultiVariable} and Section \ref{LipschitzConjugate}. For example, for the drawing
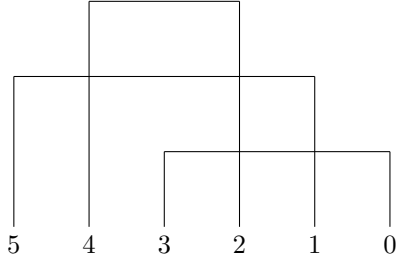
\begin{center}
    \crossing{5}{3}{0/3,1/5,2/4}{},
\end{center}
the corresponding term in $D_i e_{j_5 j_4 j_3 j_2 j_1}$ is given by
$$q_{j_1 j_2}^2q_{i j_2}q_{i j_1} \delta_{i j_3} \delta_{j_1 j_5} \delta_{j_2j_4} e_0.$$

This is equal to $q_{j_5 j_4}^2q_{i j_4}q_{i j_5} \delta_{i j_3} \delta_{j_1 j_5} \delta_{j_2j_4} e_0 $ and $q_{j_2 j_1}^2q_{j_2 i}q_{j_1 i} \delta_{i j_3} \delta_{j_1 j_5} \delta_{j_2j_4} e_0$ since $\delta_{i j}$ is the Kronecker's delta and $(q_{ij})_{1 \le i,j \le d}$ is symmetric. Therefore this term depends only on the crossings. For each $\pi \in B(n+1)$, we denote by $q^{\mathrm{cross}(\pi)} \delta_{p(\pi)}$ the coefficient as above. Then we have the same formula for a normalized dual system for mixed $q_{ij}$-Gaussians as in Proposition \ref{formula}.
Similarly, $\partial_i e_{j_n j_{n-1}\cdots j_1}$ is characterized by $C(n+1)$ in Section \ref{LipschitzConjugate} and we count crossings as above, except for crossings of right singletons and the pair that includes $0$. 

We can also derive the conjugate system of mixed $q_{ij}$-Gaussians from this combinatorics. As in the proof of Theorem \ref{conjugatesystem}, we separate the crossings into two sets, the left area and the right area. Moreover, crossings in the left area correspond to the number of inversions of permutations which are induced by pair partitions, while crossings in the right area are independent of the choices of the pair partitions. By the same arguments as in the $q$-case we get the
following formula for the conjugate system $(\xi_1,\ldots,\xi_d)$ for mixed $q_{ij}$-Gaussians $(A_1,\ldots,A_d)$:
$$\xi_i = \sum_{w \in [d]^*} (-1)^{|w|} q(w) r_{iw}^* e_w,$$ where
$$q(w)= \prod_{\substack{1\le k \le m \\ 0 \le l \le k-1}} q_{j_k j_l}\qquad\text{for}\qquad w=j_{m}\cdots j_{1}.$$
Moreover, we can extend Lemma \ref{FreeRightAnnihilation} to the $q_{ij}$-setting since Theorem 1 in \cite{Bozejko98} includes the $q_{ij}$-case. We also have Haagerup's inequality for the $q_{ij}$-setting, according to Theorem 26 in \cite{Krolak02}.

The factor $q(w)$ replaces now the factor $q^{m(m+1)/2}$, which was in the end responsible for the uniform convergence of all appearing power series expansions. 
As a consequence, if $\max_{i,j\in[d]}|q_{ij}| < 1$, then all our estimates work in the same way and we get thus that also the Lipschitz conjugate system for the mixed $q_{i j}$-Gaussians exists. 
\end{rem}

\section{Power series expansions of the conjugate variables and free Gibbs potential}\label{FreeGibbsPotential}

Finally, we also want to address estimates for the conjugate system in terms of noncommutative power series in the operators. This is relevant if we want to find a potential, such that our $q$-distribution is the corresponding free Gibbs state. 


In order to write the conjugate system as such noncommutative power series, we need to represent $r_{iw}^*e_w$ ($i \in [d]$ and $w \in [d]^*$) as a noncommutative polynomial. 
In Theorem 3.1 of \cite{EP03}, one can find the concrete formula of $Q[w]$ (see Section \ref{Preliminaries} for the definition of $Q[w]$). Here, we present this formula, actually its extension for mixed $q_{ij}$-Gaussians, by using our combinatorics. 

Let $D(n)$ be the set of partitions on $n$ vertices $n>n-1>\cdots>1$ which consist of either singletons or pair partitions. We count crossings of $\pi \in D(n)$ by using hight as in Sections \ref{MultiVariable} and \ref{LipschitzConjugate}. 
\begin{prop}\label{Effros-Popa}
Consider the setting and the notations for the $q_{ij}$-Gaussians as in Remark \ref{qijcase}. Then,
for $j_1,\ldots,j_n \in [d]$, we have
$$e_{j_n\cdots j_1} = \sum_{\pi \in D(n)}(-1)^{|p(\pi)|} q^{\mathrm{cross}(\pi)} \delta_{p(\pi)}  A^{s(\pi)}e_0$$
where $A^{s(\pi)}=A_{j_{k_s}}\cdots A_{j_{k_1}}$ for $s(\pi)=\{k_s>\cdots>k_1\}$.
\end{prop}
\begin{proof}
We prove this formula by induction on $n$. For $n=1$ it just says
$e_j=(-1)^0 q^0 A_je_0$, which is clearly true. So assume we know it for $n\geq 1$ and let us prove it for $n+1$. We have

\begin{eqnarray*}
 e_{j_{n+1}\cdots j_1}&=& A_{j_{n+1}}e_{j_n \cdots j_1} - \sum_{k=1}^n \delta_{j_{n+1}j_k}q_{j_{n+1} j_n}\cdots q_{j_{n+1} j_{k+1}} e_{j_n \cdots \check{j_k} \cdots j_1}\\
 &=& \sum_{\pi \in D(n)}(-1)^{|p(\pi)|} q^{\mathrm{cross}(\pi)} \delta_{p(\pi)} A^{j_{n+1} s(\pi)}e_0\\
 & \ & - \sum_{k=1}^n \delta_{j_{n+1}j_k}\sum_{\sigma \in D(n-1)} q_{j_{n+1} j_n}\cdots q_{j_{n+1} j_{k+1}} (-1)^{|p(\sigma)|} q^{\mathrm{cross}(\sigma)} \delta_{p(\sigma)} A^{s(\sigma)}e_0,
 \end{eqnarray*}
 where $\sigma \in D(n-1)$ acts on $j_n \cdots \check{j_k} \cdots j_1$.
 The first term corresponds to $\tilde{\pi} \in D(n+1)$ such that $\tilde{\pi}(n+1)$ is a singleton. For the second term we take,
 for each $k \in \{1,\ldots,n\}$ and $\sigma \in D(n-1)$, the $\tilde{\sigma} \in D(n+1)$ such that $n+1$ is connected to $k$ and $\tilde{\sigma} \setminus (n+1,k)=\sigma$. Then we have
 \begin{eqnarray*}
 |p(\tilde{\sigma})|= |p(\sigma)|+1 
 \qquad\text{and}\qquad
 q^{\mathrm{cross}(\tilde{\sigma})}\delta_{p(\tilde{\sigma})}= q_{j_{n+1}j_n} \cdots q_{j_{n+1}j_{k+1}} q^{\mathrm{cross}(\sigma)}\delta_{p(\sigma)}\delta_{j_{n+1}j_k}. 
 \end{eqnarray*}
 Thus the second term corresponds to such $\tilde{\sigma} \in D(n+1)$ and we have
 \begin{eqnarray*}
 e_{j_{n+1}\cdots j_1} &=& \sum_{\substack{\tilde{\pi} \in D(n+1)\\ \tilde{\pi}(n+1)\ \mathrm{is \ a \ singleton}}}(-1)^{|p(\tilde{\pi})|} q^{\mathrm{cross}(\tilde{\pi})} \delta_{p(\tilde{\pi})} A^{ s(\tilde{\pi})}e_0\\
 &\ & + \sum_{\substack{\tilde{\sigma} \in D(n+1)\\ \tilde{\sigma}(n+1)\ \mathrm{is \ not \ a \ singleton}}}(-1)^{|p(\tilde{\sigma})|} q^{\mathrm{cross}(\tilde{\sigma})} \delta_{p(\tilde{\sigma})} A^{ s(\tilde{\sigma})}e_0 \\
 &=& \sum_{\pi \in D(n+1)}(-1)^{|p(\pi)|} q^{\mathrm{cross}(\pi)} \delta_{p(\pi)}  A^{s(\pi)}e_0.
 \end{eqnarray*}
\end{proof}
Using this we can rewrite our conjugate variables as noncommutative power series
in $A_1,\ldots,A_d$. (In the following we will, for simplicity, again restrict to the $q$-case, though the $q_{ij}$-case can be treated in the same way.) The main point will be to see that we have good estimates for the operator norms of the summands in these series; this will be similar to the proof of Corollary \ref{Lipschitz_for_q_Gaussian}.
 Let us fix $i \in [d]$. For each $w \in[d]^*$ with $|w|=m$, we write $r_{iw}^* e_w =\sum_{|v|= 2m+1}\alpha_v e_v$ (as before we suppress in the notation for $\alpha_v$ the dependency on $w$). Recall that we have $|\alpha_v| \le \sqrt{w(q)}^{-3m-2}\sqrt{[m]_{q}!}$ for any $|v|=2m+1$ (see the proof of Corollary \ref{Lipschitz_for_q_Gaussian}).
 Moreover, by Proposition \ref{Effros-Popa}, we have
 $$e_v = \sum_{\pi \in D(2m+1)}(-1)^{|p(\pi)|} q^{\mathrm{cross}(\pi)} \delta_{p(\pi)}  A^{s(\pi)}e_0.$$
 
 Then we have
 \begin{eqnarray*}
 \xi_i&=&\sum_{m=0}^{\infty} (-1)^{m}q^{\frac{m(m+1)}{2}} \sum_{|w|=m}r_{iw}^* e_w\\
&=&\sum_{m=0}^{\infty} (-1)^{m}q^{\frac{m(m+1)}{2}}\sum_{|w|=m}\sum_{|v|=2m+1} \alpha_v  e_v\\
&=& \sum_{m=0}^{\infty} (-1)^{m}q^{\frac{m(m+1)}{2}}\sum_{|w|=m}\sum_{|v|=2m+1}  \alpha_v \sum_{\pi \in D(2m+1)} (-1)^{|p(\pi)|} q^{\mathrm{cross}(\pi)}\delta_{p(\pi)}A^{s(\pi)}e_0. 
\end{eqnarray*}
This is our ''concrete'' realization for the $\xi_i$ as non-commutative power series in $A_1,\dots,A_d$. We claim that these power series have infinite radius of convergence. 
We set $A=\max_{i\in[d]} \|A_i\|> 1$.  Then we can estimate the operator norm as follows:
\begin{eqnarray*}\Vert \xi_i\Vert
&\leq &\sum_{m=0}^{\infty} |q|^{\frac{m(m+1)}{2}}\sum_{|w|=m}\sum_{|v|=2m+1}  |\alpha_v| \sum_{\pi \in D(2m+1)} |q^{\mathrm{cross}(\pi)}\delta_{p(\pi)}| \|A^{s(\pi)}\| \\
&\le& \sum_{m=0}^{\infty} |q|^{\frac{m(m+1)}{2}}\sum_{|w|=m}\sum_{|v|=2m+1} |\alpha_v| (2m+1)! A^{2m+1} \\
&\le& \sum_{m=0}^{\infty} |q|^{\frac{m(m+1)}{2}}\sum_{|w|=m}\sum_{|v|=2m+1}  
\sqrt{w(q)}^{-3m-2} \sqrt{[m]_q!} (2m+1)! A^{2m+1} \\
&=& \sum_{m=0}^{\infty} |q|^{\frac{m(m+1)}{2}} \left(\frac{d}{\sqrt{w(q)}}\right)^{3m+2} \sqrt{[m]_q!} (2m+1)! A^{2m+1}
\end{eqnarray*}
where we use $|D(2m+1)| \le (2m+1)!$, since all partitions in $D(2m+1)$ have blocks of size either 1 or 2 and can thus be identified with permutations in the symmetric group of degree $2m+1$. 

By the ratio test, this sum converges for any $A$ and thus this implies that the conjugate system is a $d$-tuple of noncommutative power series which are uniformly convergent with radius of convergence equal to $\infty$. 

A free Gibbs potential (see Section 1.2 in \cite{GS14}) for $q$-Gaussians is an operator $V \in W^*(A_1,\ldots,A_d)$ which satisfies $\mathcal{D}_i V = \xi_i$ for all $i \in [d]$ where 
$\mathcal{D}_i$'s are the cyclic derivatives defined by $\mathcal{D}_i = m_{\mathrm{flip}} \circ \partial_i$ ($m_{\mathrm{flip}}$ is defined by $ m_{\mathrm{flip}}(a \otimes b)=b a$).
When we write the conjugate system as noncommutative power series $\xi_i = \sum_{w \in [d]^*} \alpha(w,i) A^w$, this potential $V$ is formally given by (see the proof of Corollary 4.3 in \cite{GS14})
$$ V= \frac{1}{2} N^{-1}\left(\sum_{i=1}^d A_i \xi_i+\xi_i A_i\right) = \sum_{i=1}^d \sum_{w \in [d]^*} \frac{\alpha(w,i)}{2(1+|w|)} (A^{iw}+A^{wi})$$
where $N$ is the number operator which maps $A^w$ to $|w|A^w$. Estimates as above tell us the uniform convergence of the noncommutative power series on the right hand side, yielding the existence of a free Gibbs potential.
 \begin{prop}
 A free Gibbs potential exists for the  $q$-Gaussians, for any $-1<q<1$.
 \end{prop}


\begin{thebibliography}{99}

\bibitem{Avsec11}
S.~{Avsec}.  
\newblock{Strong Solidity of the q-Gaussian Algebras for all $-1 < q < 1$.}
\newblock arXiv: 1110.4918. (2011).

\bibitem{BM20}
M.~{Banna}, T.~{Mai}.
\newblock{Hölder continuity of cumulative distribution functions for noncommutative polynomials under finite free Fisher information.}
\newblock{\em{Journal of Functional Analysis}} 279 (2020).

\bibitem{Bozejko98}
M.~{Bożejko}. 
\newblock{Completely positive maps on Coxeter groups and the ultracontractivity of the $q$-Ornstein-Uhlenbeck semigroup.} 
\newblock{\em{Banach Center Publications}} 43.1 (1998): 87-93. 

\bibitem{Bozejko99}
M.~{Bożejko}.
\newblock Ultracontractivity and strong Sobolev inequality for  $q$-Ornstein-Uhlenbeck-semigroup ($-1 < q < 1$)
\newblock Infinite Dimensional Analysis, Quantum Probability and Related Topics, 1999 02:02, 203-220.

\bibitem{BS91}
M.~{Bożejko},R.~{Speicher}. 
\newblock{An example of a generalized Brownian motion.}
\newblock{\em{Commun.Math. Phys.}} 137, 519–531 (1991).

\bibitem{BS94}
M.~{Bożejko}, R.~{Speicher}. 
\newblock Completely positive maps on Coxeter groups, deformed commutation relations, and operator spaces. 
\newblock {\em{Math. Ann.}} 300, 97–120 (1994).

\bibitem{BKS97}
M.~{Bożejko}, B.~{Kümmerer}, and R.~{Speicher}. 
\newblock q-Gaussian Processes: Non-commutative and Classical Aspects. 
\newblock{\em{Commun. Math. Phys.}} 185, 129–154 (1997). 

\bibitem{CI20}
H.~{Cohl}, M.~{Ismail}.  
\newblock{Lectures on Orthogonal Polynomials and Special Functions (London Mathematical Society Lecture Note Series).}
\newblock Cambridge: Cambridge University Press (2020).

\bibitem{Dabrowski10}
Y.~{Dabrowski}. 
\newblock{A free stochastic partial differential equation.}
\newblock{\em{Annales De L Institut Henri Poincare-probabilites Et Statistiques}} 50 (2010): 1404-1455.

\bibitem{Dabrowski10JFA}
Y.~{Dabrowski}. 
\newblock{A note about proving non-$\Gamma$ under a finite non-microstates free Fisher information assumption.}
\newblock{\em{Journal of Functional Analysis}} 258 (2010): 3662-3674.

\bibitem{DL16}
Y.~{Dabrowski}, A.~{Ioana}.
\newblock{Unbounded derivations, free dilations, and indecomposability results for II$_1$ factors.}
\newblock{\em{Trans. Amer. Math. Soc.}} 368 (2016), 4525-4560.

\bibitem{EP03}
E.G.~{Effros}, M.~{Popa}
\newblock {Feynman diagrams and Wick products associated with $q$-Fock space.}
\newblock National Academy of Sciences. 100 (15) 8629-8633 (2003).

\bibitem{GS14}
A.~{Guionnet}, D.~{Shlyakhtenko}. 
\newblock{Free monotone transport.}
\newblock{\em{Invent. math.}} 197, 613–661 (2014).

\bibitem{JSW}
P.E.T. Jorgensen, L.M. Schmitt, and R.F. Werner. 
\newblock{Positive representations of general commutation relations allowing Wick ordering.} 
\newblock{\em Journal of Functional Analysis} 134.1, 33-99 (1995).

\bibitem{Krolak02}
I.~{Kr\.{o}lak}  
\newblock {Haagerup inequality and contractivity properties of Ornstein-Uhlenbeck semigroup for general commutation relations.} 
\newblock Polish Academy of Sciences. Institute of Mathematics; 2002.

\bibitem{MSY18}
T.~{Mai}, R.~{Speicher} and S.~{Yin}.
\newblock The free field: zero divisors, Atiyah property and realizations via unbounded operators.
\newblock {\em arXiv preprint arXiv:1805.04150v2}, 2018. 

\bibitem{MSY19}
T.~{Mai}, R.~{Speicher} and S.~{Yin}.
\newblock The free field: realization via unbounded operators and Atiyah property.
\newblock {\em arXiv preprint arXiv:1905.08187}, 2019. 

\bibitem{MS17}
J.~A. {Mingo} and R.~{Speicher}.
\newblock {\em Free probability and random matrices}, volume~35 of {\em Fields
  Institute Monographs}.
\newblock Springer, New York; Fields Institute for Research in Mathematical  Sciences, Toronto, ON, 2017.

\bibitem{Nou04}
A.~{Nou}. 
\newblock{Non injectivity of the $q$-deformed von Neumann algebra.}
\newblock{\em{Mathematische Annalen}} 330 (2004): 17-38.

\bibitem{PS21}
M.~{Pluma}, R.~{Speicher}. 
\newblock{A Dynamical Version of the SYK Model and the $q$-Brownian Motion.} 
\newblock Random Matrices: Theory and Applications (2021).

\bibitem{Ricard05}
\'{E}.~{Ricard}
\newblock{Factoriality of $q$-Gaussian von Neumann Algebras.}
\newblock{\em{Commun. Math. Phys.}} 257, 659–665 (2005).

\bibitem{Shlyakhtenko04}
D.~{Shlyakhtenko}. 
\newblock{Some estimates for non-microstates free entropy dimension with applications to q-semicircular families.}
\newblock{International Mathematics Research Notices.} Volume 2004, 51, 2757–2772 (2004).

\bibitem{Shlyakhtenko06}
D.~{Shlyakhtenko}.  
\newblock{Remarks on Free Entropy Dimension.}
\newblock{In: Bratteli O., Neshveyev S., Skau C. (eds) Operator Algebras. Abel Symposia, vol 1. Springer, Berlin, Heidelberg} (2006). 

\bibitem{Sniady01}
P.~{\'{S}niady}. 
\newblock{Gaussian Random Matrix Models for q-deformed Gaussian Variables.} 
\newblock{\em{Commun. Math. Phys.}} 216, 515–537 (2001).

\bibitem{Sniady04}
P.~{\'{S}niady}. 
\newblock{Factoriality of Bożejko–Speicher von Neumann Algebras.}
\newblock{\em{Commun. Math. Phys.}} 246, 561–567 (2004). 

\bibitem{Sp93}
R. Speicher
\newblock{Generalized statistics of macroscopic fields.} \newblock{\em Lett. Math. Phys.} 27.2, 97-104 (1993).

\bibitem{Voi98}
D.~{Voiculescu}. 
\newblock{The analogues of entropy and of Fisher's information measure in free probability theory V. Noncommutative Hilbert Transforms.}
\newblock{\em{Inventiones mathematicae}} 132, 189-227 (1998).


\end{thebibliography}
\end{document}